\documentclass[12pt]{article}
\usepackage[utf8]{inputenc}
\usepackage[T1]{fontenc}
\usepackage[english]{babel}
\usepackage{amsmath,amssymb,amsfonts,amsthm,mathptmx}
\usepackage[inner=2.4cm,outer=2.65cm,bottom=3cm]{geometry}
\usepackage{color}
\usepackage{soul}
\usepackage{hyperref}
\usepackage[autostyle]{csquotes}
\usepackage{mathrsfs}
\usepackage{pgfplots}
\usepackage{nicefrac}
\usepackage{booktabs}
\usepackage{bbm}
\usepackage{tikz}
\usepackage{subcaption}
\usepackage{listings}
\usepackage{enumitem}

\usepackage{xcolor}
\definecolor{codegreen}{rgb}{0,0.6,0}
\definecolor{codegray}{rgb}{0.5,0.5,0.5}
\definecolor{codeorange}{rgb}{1,0.49,0}
\definecolor{backcolour}{rgb}{0.95,0.95,0.96}

\lstdefinestyle{mystyle}{
    backgroundcolor=\color{backcolour},   
    commentstyle=\color{codegray},
    keywordstyle=\color{codeorange},
    numberstyle=\tiny\color{codegray},
    stringstyle=\color{codegreen},
    basicstyle=\ttfamily\footnotesize,
    breakatwhitespace=false,         
    breaklines=true,                 
    captionpos=b,                    
    keepspaces=true,                 
    numbers=left,                    
    numbersep=5pt,                  
    showspaces=false,                
    showstringspaces=false,
    showtabs=false,                  
    tabsize=2,
    xleftmargin=10pt,
}

\usetikzlibrary{intersections,pgfplots.fillbetween}

\newcommand{\uproman}[1]{\uppercase\expandafter{\romannumeral#1}}
\theoremstyle{remark}
\newtheorem{remark}{Remark}[section]
\theoremstyle{definition}
\newtheorem{theorem}{Theorem}[section]

\newtheorem{example}[theorem]{Example}

\newtheorem{algorithm}[theorem]{Algorithm}
\newtheorem{gessaman rule}[theorem]{Gessaman's rule}
\newtheorem{BTC-rule}[theorem]{BTC-rule}
\newtheorem*{contiex1}{Continuation of Example 1.2}
\DeclareMathOperator{\esssup}{ess-sup}

\allowdisplaybreaks
\usepackage{graphicx}

\DeclareMathAlphabet{\mathcal}{OMS}{cmsy}{m}{n}

\numberwithin{equation}{section}

\newcommand{\q}[1]{`#1'}

\newcommand\restr[2]{{
  \left.\kern-\nulldelimiterspace 
  #1 
  \littletaller 
  \right|_{#2} 
  }}

\newcommand{\littletaller}{\mathchoice{\vphantom{\big|}}{}{}{}}

\title{Data-dependent density estimation for the Fokker-Planck equation in higher dimensions}

\author{Max Jensen\footnote{Mathematics Department, University College London, 25 Gordon Street, London WC1H 0AY, United Kingdom. email: {\tt max.jensen@ucl.ac.uk}} \qquad
Fabian Merle\footnote{Mathematisches Institut,
Universit\"{a}t T\"{u}bingen,
Auf der Morgenstelle 10,
D-72076 T\"{u}bingen, Germany. email: {\tt merle@na.uni-tuebingen.de}} \qquad
Andreas Prohl\footnote{
Mathematisches Institut,
Universit\"{a}t T\"{u}bingen,
Auf der Morgenstelle 10,
D-72076 T\"{u}bingen, Germany. email: {\tt prohl@na.uni-tuebingen.de}}}

\date{Version of the draft: \today}

\begin{document}

\maketitle
\begin{abstract}
We present a new strategy to approximate the global solution of the {\em Fokker-Planck equation} efficiently in higher dimension and show its convergence. The main ingredients are the Euler scheme to solve the associated stochastic differential equation and a histogram method for tree-structured {\em density estimation} on a {\em data-dependent} partitioning of the state space $\mathbb{R}^{d}$.

\medskip


\medskip

\end{abstract}


\section{Introduction}
\label{1}

Let $d\in \mathbb{N}$, fix $T>0$, and choose a probability density function (pdf, for short) $p_{0}:\mathbb{R}^{d}\to [0,\infty)$ with $\lVert p_{0} \rVert_{\mathbb{L}^{1}}=1$. In this work, we propose and study a new method that uses probabilistic and statistical tools to approximate the {\em global} solution $\big\{p(t,\cdot)\,;\, t\in [0, T]\big\}$ of the {\em Fokker-Planck equation}
\begin{equation}
\label{eq:fpe}
\begin{split}
\partial_{t}p(t,\mathbf{x}) &= \mathcal{L}^{\ast}p(t,\mathbf{x}) \quad \text{for all } \;\, (t,{\bf x})\in (0,T]\times \mathbb{R}^{d}, \\ 
p(0,\mathbf{x})&=p_{0}(\mathbf{x}) \quad \text{for all} \;\, \mathbf{x}\in \mathbb{R}^{d}\,,
\end{split}
\end{equation}
in higher dimensions $d$, where 
\begin{equation}
\label{eq:operator}
\mathcal{L}^{\ast}p(t,\mathbf{x})=-\sum\limits_{i=1}^{d} \partial_{x_{i}}\left(b_{i}(\mathbf{x})\cdot p(t,\mathbf{x})\right) + \frac{1}{2}\sum\limits_{i,j,k=1}^{d}\partial_{x_{i}}\partial_{x_{j}}\left( \sigma_{ik}(\mathbf{x})\cdot\sigma_{jk}(\mathbf{x})\cdot p(t,\mathbf{x})\right)\,
\end{equation}
is a second-order differential operator, in general, with non-constant coefficients. Under the assumptions stated in Subsection \ref{2.1}, there exists a unique classical solution $p:[0, T]\times \mathbb{R}^{d}\to [0,\infty)$ of \eqref{eq:fpe} which coincides with the pdf of the related stochastic differential equation (SDE) process $\mathbf{X}\equiv\big\{ \mathbf{X}_{t}\,;\, t \in [0, T]\big\}$, which is given by
\begin{equation}
\label{eq:SDE}
\begin{split}
\mathrm{d} \mathbf{X}_{t} &= \mathbf{b}(\mathbf{X}_{t})\mathrm{d}t + \pmb{\sigma}(\mathbf{X}_{t})\mathrm{d}\mathbf{W}_{t} \qquad \text{for all }t\in (0,T] \\ 
\mathbf{X}_{0}&= \mathbf{X}^{0}\,;
\end{split}
\end{equation}
see e.g.~\cite[Prop.~3.3]{pavliotis} and \cite[Chap.~2, Thm.~1.1]{freidlin}. Here, $\mathbf{b}:\mathbb{R}^{d}\to \mathbb{R}^{d}$ and $\pmb{\sigma}:\mathbf{R}^{d}\to \mathbb{R}^{d\times d}$ are smooth bounded functions with entries $b_{i}(\cdot)$ resp.~$\sigma_{ij}(\cdot)$, and $\mathbf{W}\equiv\{\mathbf{W}_{t}\,;\,t\in [0,T]\}$ is a $\mathbb{R}^{d}-$valued Wiener process on a filtered probability space $\bigl(\Omega,\mathcal{F},\{\mathcal{F}_{t}\}_{t\geq 0},\mathbb{P}\bigr)$, and $\mathbf{X}^{0}$ is random and drawn according to $p_{0}$.

\medskip

To numerically solve \eqref{eq:fpe} in higher dimensions $d\gg 1$, different {\em deterministic} and {\em probabilistic} methods have been developed in the last two decades and are the subject of active research. For example, well-known deterministic methods to simulate \eqref{eq:fpe} in dimensions $d\gg 1$, e.g.~include \q{sparse grid} methods \cite{bungartz}, and \q{low-rank approximation} methods \cite{bachmayr1, bachmayr2}, but their efficient use (and theory) is restricted to \q{tensor-sparsity respecting data} $\mathcal{L}^{\ast}$ and $p_{0}$ in \eqref{eq:fpe}; see the recent survey in \cite[Sec.~2]{fabi2}. A different numerical strategy is to use the probabilistic reformulation \eqref{eq:SDE} as a starting point for discretisation and statistical sampling via Monte-Carlo simulations; see e.g.~\cite{gobet}; while the analysis there is exempted from the assumption on \q{tensor-sparsity respecting data}, its straightforward application only yields approximations of values for $p$ from \eqref{eq:fpe} at selected space-time points. This work aims at constructing numerical methods for \eqref{eq:fpe} and $d\gg 1$ that are based on the probabilistic reinterpretation \eqref{eq:SDE} and approximate the \q{entire} function $p(T,\cdot):\mathbb{R}^{d}\to (0,\infty)$.

\medskip

From a practical viewpoint, the new Algorithm \ref{strategy} below in this work stands out due to its simple implementability, its intuitive interpretability, and efficiency; from a theoretical viewpoint, its convergence theory covers {\em non-constant} $\mathcal{L}^{\ast}\equiv \mathcal{L}^{\ast}(\mathbf{x})$ in higher dimensions $d\gg 1$, which allows broader potential applicability, including the simulation of polymer models, for example, see e.g.~\cite{barrett1,barrett2}. Conceptually, its construction combines
\begin{itemize}
\item[a)] 
the (explicit) Euler scheme for \eqref{eq:SDE} on a mesh $\{t_{j}\}_{j=0}^{J}\subset [0,T]$ of fixed step size $\tau=\frac{T}{J}$, $J\in \mathbb{N}$, which determines the $\mathbb{R}^{d}-$valued random variable $\mathbf{Y}^{j+1}$ from
\begin{equation}
\label{eq:euler}
\begin{split}
\mathbf{Y}^{j+1}&=\mathbf{Y}^{j}+\tau \cdot \mathbf{b}(\mathbf{Y}^{j}) + \pmb{\sigma}(\mathbf{Y}^{j})(\mathbf{W}_{t_{j+1}}-\mathbf{W}_{t_{j}})\qquad (j=0,...,J-1)\\
\mathbf{Y}^{0}&=\mathbf{X}^{0}
\end{split}
\end{equation}
to approximate $\mathbf{X}_{t_{j+1}}$ from \eqref{eq:SDE} at time $t_{j+1}= t_{j}+ \tau$. The density $p^{J}$ of the induced measure $\mathbb{P}_{\mathbf{Y}^{J}}$ approximates $p(T,\cdot)$ from \eqref{eq:fpe}; see Theorems \ref{ballythm} and \ref{modballythm}.
\item[b)] a {\em data-dependent} partitioning  $\pmb{\mathcal{P}}^{J,\mathtt{M}}$ of $\mathbb{R}^d$ based on an $\mathtt{M}$-sample $\mathbf{D}^{J}_{\mathtt{M}}:= \{ \mathbf{Y}^{J,m}\}_{m=1}^{\mathtt{M}}$ from \eqref{eq:euler}, and {\em statistically equivalent cells} \cite[Ch.~13]{gyoerfi2} on which the histogram-based estimator $\widehat{\mathtt{p}}^{J,\mathtt{M}}$ for the pdf $p^{J}$ from \eqref{eq:euler} at time $t_{J}=T$ is constructed; see Subsection \ref{3.1}.
\end{itemize}

Such an estimator $\widehat{\mathtt{p}}^{J,\mathtt{M}}$ is advantageous for several reasons:
\begin{enumerate}
\item[(a)] suitability to approximate solutions of \eqref{eq:fpe} for $d \gg 1$, where uniform partitioning strategies tend to exhaust even extensive computational resources in the pre-asymptotic range and therefore lose competitivity.
\item[(b)] automatic construction of {\em adaptive} meshes, where locally appearing finer cells $\mathbf{R}_{r}\in \pmb{\mathcal{P}}^{J,\mathtt{M}}$ adjust to local features of the solution $p^J:\mathbb{R}^{d}\to [0,\infty)$. The  proposed \q{{\em data-dependent}} estimators
$\widehat{\mathtt{p}}^{J,\mathtt{M}}:\mathbb{R}^{d}\to [0,\infty)$ below aim to meet two goals at the same time with the help of a single $\mathtt{M}-$sample
$\mathbf{D}^{J}_{\mathtt{M}}$: to generate an {\em adaptive} mesh $\pmb{\mathcal{P}}^{J,\mathtt{M}}$, and to compute an accurate approximation $\widehat{\mathtt{p}}^{J,\mathtt{M}}$ of $p^J$ on that mesh. In approximation theory, generating an underlying partition which depends on the solution itself is referred to as \q{nonlinear approximation} \cite[Section 3.2]{DV};
well-known deterministic methods in this direction are  \q{adaptive finite elements} that are applicable in
small dimensions $1 \leq d \leq 4$, where criteria trigger a repeated automatic local refinement/coarsening of elements with the help of computed local residuals that leads to a final adapted mesh.
\end{enumerate}
In the vein of (b), Algorithm \ref{strategy} below may be seen as a new probabilistic method for the parabolic problem \eqref{eq:fpe}, which is adaptive in space and whose convergence for general non-constant operators $\mathcal{L}^{\ast}\equiv \mathcal{L}^{\ast}(\mathbf{x})$ follows from Theorem \ref{thmgessamanconvergence} in Subsection \ref{3.2}. Moreover, the Algorithm \ref{strategy} is applicable in higher dimensions $d\gg 1$ -- a property that adaptive finite element methods do not possess. 

\medskip

We now outline the basic steps of our algorithm to solve \eqref{eq:fpe} with the help of the tools that were motivated in (a) and (b) above.

\begin{algorithm}
\label{strategy}
Let $J\in \mathbb{N}$. Fix an equidistant mesh $\{t_{j}\}_{j=0}^{J}\subset [0,T]$ of mesh-size $\tau=\frac{T}{J}$.
\begin{itemize}
\item[1)] {\bf (Initialization)} For $\mathtt{M}\in \mathbb{N}$, generate an {\em i.i.d.}~$\mathtt{M}-$sample $\mathbf{D}_{\mathtt{M}}^{0}:=\{\mathbf{Y}^{0,m}\}_{m=1}^{\mathtt{M}}$ through sampling via $p_{0}$.
\item[2)] {\bf (Discretization \& Sampling)} Generate a $\mathtt{M}-$sample $\mathbf{D}_{\mathtt{M}}^{J}:=\{\mathbf{Y}^{J,m}\}_{m=1}^{\mathtt{M}}$ via {\em i.i.d.}~sample iterates from \eqref{eq:euler}.
\item[3)] {\bf (Splitting rule)} Use $\mathbf{D}_{\mathtt{M}}^{J}$ to get a {\em data-dependent} partitioning $\pmb{\mathcal{P}}^{J,\mathtt{M}}=\{\mathbf{R}_{r};\,r=0,...,\mathtt{R}-1\}$ of $\mathbb{R}^{d}$ into $\mathtt{R}\in \mathbb{N}$ many cells \q{$\mathbf{R}_{r}$}, where each cell contains $k_{\mathtt{M}}\leq \mathtt{M}$ sample iterates.
\item[4)] {\bf (Estimator)} Define
\begin{equation}
\label{eq:estimator}
\widehat{\mathtt{p}}^{J,\mathtt{M}}(\mathbf{x}):=\sum\limits_{r=0}^{\mathtt{R}-1} \frac{k_{\mathtt{M}}}{\mathtt{M}\cdot \mathtt{vol}(\mathbf{R}_{r})}\cdot \mathbbm{1}_{\{\mathbf{x}\in \mathbf{R}_{r}\}}\,,\qquad \mathbf{x}\in \mathbb{R}^{d}\,,
\end{equation}
where $\mathtt{vol}(\mathbf{R}_{r})$ denotes the (possibly infinite) volume of the cell $\mathbf{R}_{r}$.
\end{itemize}
\end{algorithm}

Two specific splitting rules to generate partitions $\pmb{\mathcal{P}}^{J,\mathtt{M}}$ will be detailed in Section \ref{3}: the first rule is due to Gessaman \cite{gessaman} ({\em cf.}~Gessaman's rule \ref{gessamanstrategy}), the second one is a modification of the {\em Binary Tree Cuboid} (BTC)-splitting rule taken from \cite{dunst} ({\em cf.}~BTC-rule \ref{moddunststrategy}). Both splitting rules generate partitions of statistically equivalent cells, meaning that each cell contains the same number of realisations of $\mathbf{D}_{\mathtt{M}}^{J}$. Our main theoretical result in this work is to prove $\mathbb{P}-a.s.$ convergence of Algorithm \ref{strategy} for Gessaman's rule \ref{gessamanstrategy}; see Theorem \ref{thmgessamanconvergence}. More precisely, we prove that for $\widehat{\mathtt{p}}^{J,\mathtt{M}}\equiv \widehat{\mathtt{p}}^{J,\mathtt{M}}_{\mathtt{Ges}}$ in \eqref{eq:estimator},
\begin{align}
\label{eq:convergence}
\lVert p(T,\cdot)-\widehat{\mathtt{p}}^{J,\mathtt{M}}_{\mathtt{Ges}}\rVert_{\mathbb{L}^{1}}\to 0\qquad \mathbb{P}-a.s. \qquad (\text{for } \mathtt{M}\to \infty  \text{ and } \tau \to 0  )\,.
\end{align}

\medskip

The following example of non-constant $\mathcal{L}^{\ast}\equiv \mathcal{L}^{\ast}(\mathbf{x})$ in \eqref{eq:operator} illustrates the generation of adaptive meshes in space by Algorithm \ref{strategy}.

\begin{example}
\label{example1}
Let $d\in \mathbb{N}$ and $T>0$ initially. Consider \eqref{eq:fpe}, which corresponds to the associated SDE \eqref{eq:SDE} with
\begin{align*}
\mathbf{b}(\mathbf{x})=-\mathbf{x}+\pmb{\beta}\,, \qquad \pmb{\sigma}(\mathbf{x})=\sqrt{2}\cdot \varepsilon \cdot \pmb{Id}_{d}\,,
\end{align*}
where $\varepsilon>0$, $\pmb{\beta}=[1,...,1]^{\top}\in \mathbb{R}^{d}$, and $\pmb{Id}_{d}$ denotes the $d-$dimensional identity matrix. Let $p_{0}$ be the pdf of the multivariate normal distribution with mean vector $\mathbf{m}_{0}=[0,...,0]^{\top}\in \mathbb{R}^{d}$ and covariance matrix $\mathbf{C}_{0}=\alpha \cdot \pmb{Id}_{d}$, i.e.,
\begin{equation}
\label{p0ex1}
p_{0}(\mathbf{x})=\frac{1}{(2\pi \alpha)^{\nicefrac{d}{2}}}\cdot \exp \left(-\frac{1}{2\alpha} \cdot \lVert \mathbf{x} \rVert_{\mathbb{R}^{d}}^{2}\right)\,,
\end{equation}
where $\alpha>0$.
The true solution of \eqref{eq:fpe} is given by (see e.g.~\cite[Appendix C]{boffi})
\begin{equation}
\label{sol.ex1}
p(t,\mathbf{x})=\frac{1}{\sqrt{(2\pi)^{d} \cdot \mathtt{det}(\mathbf{C}_{t})}} \exp \left(-\frac{1}{2} ( \mathbf{x} - \mathbf{m}_{t})^{\top} \mathbf{C}_{t}^{-1} (\mathbf{x}-\mathbf{m}_{t}) \right)\qquad (t,\mathbf{x})\in [0,T]\times \mathbb{R}^{d}\,,
\end{equation}
where 
\begin{equation*}
\mathbf{m}_{t}=[1-\exp(-t),...,1-\exp(-t)]^{\top} \in \mathbb{R}^{d}\,,
\end{equation*}
and 
\begin{equation*}
\mathbf{C}_{t}=\left(\alpha \cdot \exp(-2t) + \varepsilon^{2}(1-\exp(-2t)) \right) \cdot \pmb{Id}_{d}.
\end{equation*}
{\bf a)} Now fix $d=2$, $T=0.5$, $\tau=\frac{1}{100}$, $\alpha=0.5$, and $\varepsilon=1$. For different choices of $\mathtt{M}$ and $k_{\mathtt{M}}$, Table \ref{tablegessaman} presents $\mathbb{L}^{\infty}-$errors
for the estimator $\widehat{\mathtt{p}}^{J,\mathtt{M}}\equiv \widehat{\mathtt{p}}^{J,\mathtt{M}}_{\mathtt{Ges}}$ given in \eqref{eq:estimator} constructed via Algorithm \ref{strategy} with Gessaman's rule \ref{gessamanstrategy} in Subsection \ref{3.1.1}. For ease of computation, we present errors in the $\mathbb{L}^{\infty}-$norm in all computations; the theoretical convergence results are obtained in the $\mathbb{L}^{1}-$norm.

\medskip

\begin{table}[h]
\center
\begin{tabular}{lclclclc}
\toprule
 $\mathtt{M}$ & $k_{\mathtt{M}}$ & $\lVert p(T,\cdot)-\widehat{\mathtt{p}}_{\mathtt{Ges}}^{J,\mathtt{M}}\rVert_{\mathbb{L}^{\infty}}$  & Computational time\\
\midrule
 $2^{12}$  & $2^{6}$  & $\approx 0.09$  & $3$ sec\\
 $2^{14}$  & $2^{7}$  & $\approx 0.08$ &  $11$ sec\\
 $2^{16}$  & $2^{8}$  & $\approx 0.05$  & $43$ sec\\
 $2^{18}$  & $2^{9}$  & $\approx 0.04$ &  $172$ sec\\
 $2^{20}$  & $2^{10}$ & $\approx 0.03$ & $698$ sec\\
\bottomrule
\end{tabular}
\captionof{table}{Example \ref{example1} {\bf a)}: $\mathbb{L}^{\infty}-$error for $\widehat{\mathtt{p}}^{J,\mathtt{M}}\equiv \widehat{\mathtt{p}}^{J,\mathtt{M}}_{\mathtt{Ges}} $ in \eqref{eq:estimator}.}
\label{tablegessaman}
\end{table}

{\bf b)} Now fix $d=2$, $T=1$, $\tau=\frac{1}{100}$, $\alpha=\frac{1}{4\pi}$, and $\varepsilon=\frac{1}{5}$. Figures \ref{figex1} resp.~\ref{figex1_1} show snapshots of partitions resp.~related solution profiles. We observe automatic migration/adaption of meshes in the direction given by $\pmb{\beta}$.

\begin{figure}[h!]
\begin{minipage}[t]{.25\textwidth}
\centering
\includegraphics[scale=0.35]{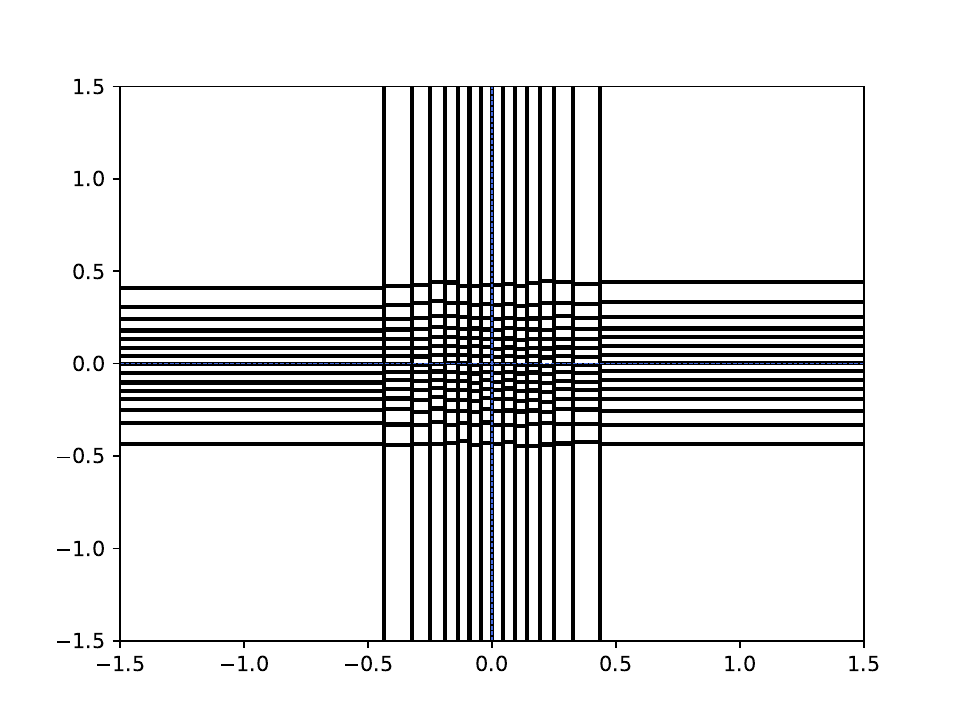}
\captionsetup{labelfont={bf}}
\subcaption{$t=0$, $j=0$}
\end{minipage}
\hspace{.08\linewidth}
\begin{minipage}[t]{.25\textwidth}
\centering
\includegraphics[scale=0.35]{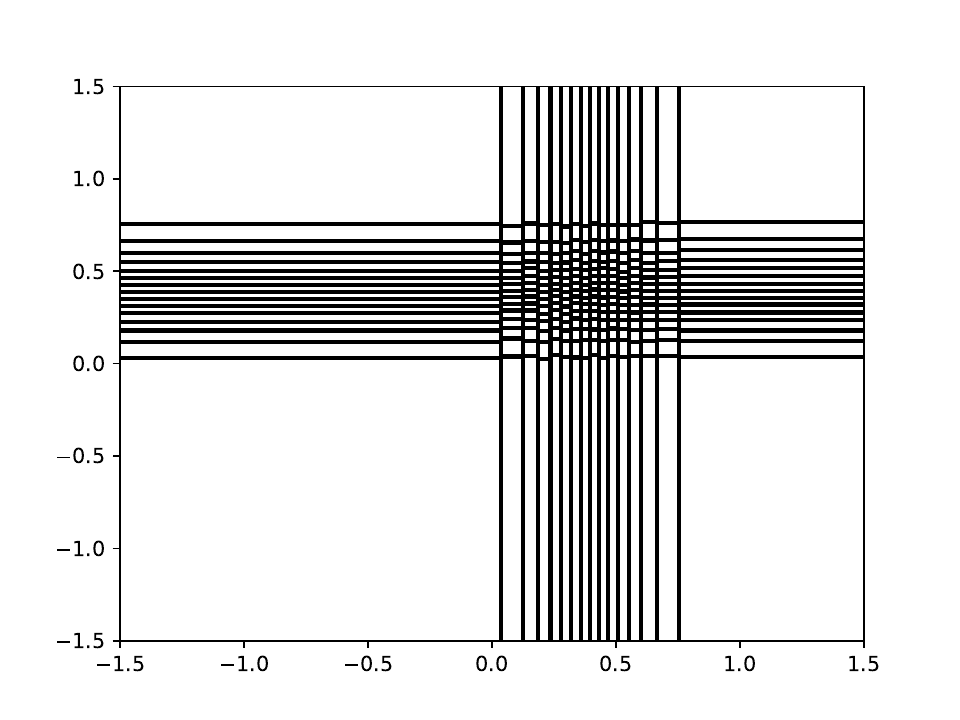}
\captionsetup{labelfont={bf}}
\subcaption{$t=0.5$, $j=50$}
\end{minipage}
\hspace{.08\linewidth}
\begin{minipage}[t]{.25\textwidth}
\centering
\includegraphics[scale=0.35]{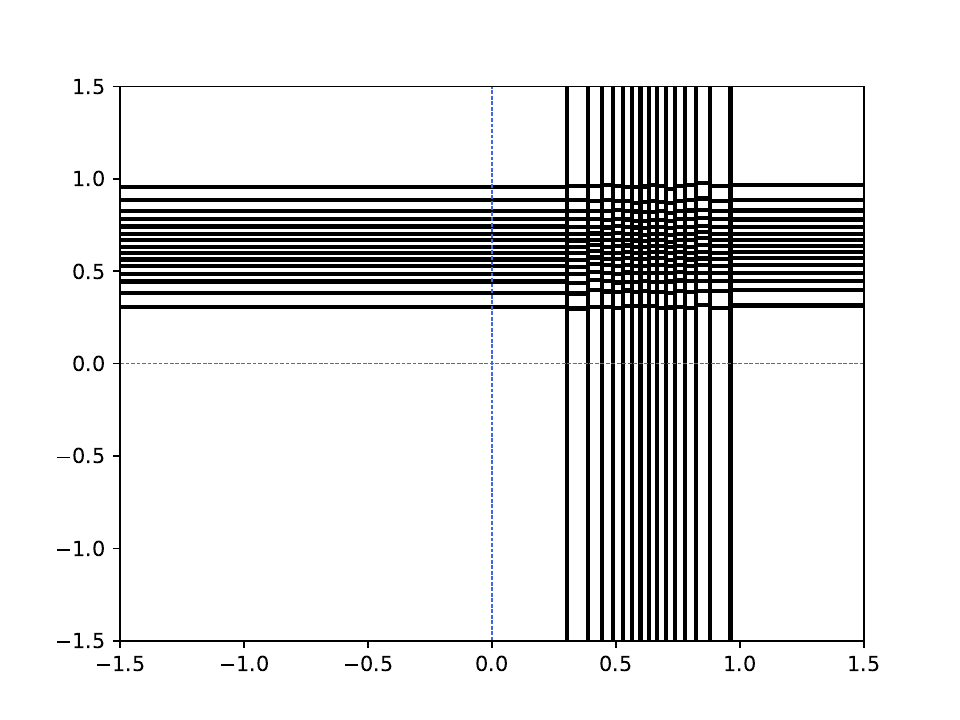}
\captionsetup{labelfont={bf}}
\subcaption{$t=1$, $j=100$}
\end{minipage}

\caption{ Example \ref{example1} {\bf b)}: Partition $\pmb{\mathcal{P}}^{j,\mathtt{M}}_{\mathtt{Ges}}$ from Algorithm \ref{strategy} with Gessaman's rule \ref{gessamanstrategy}; see Subsection \ref{3.1.1} ($\mathtt{M}=2^{16}$, $k_{\mathtt{M}}=2^{8}$).}
\label{figex1}

\end{figure}

\begin{figure}[h!]
\begin{minipage}[t]{.25\textwidth}
\centering
\includegraphics[scale=0.42]{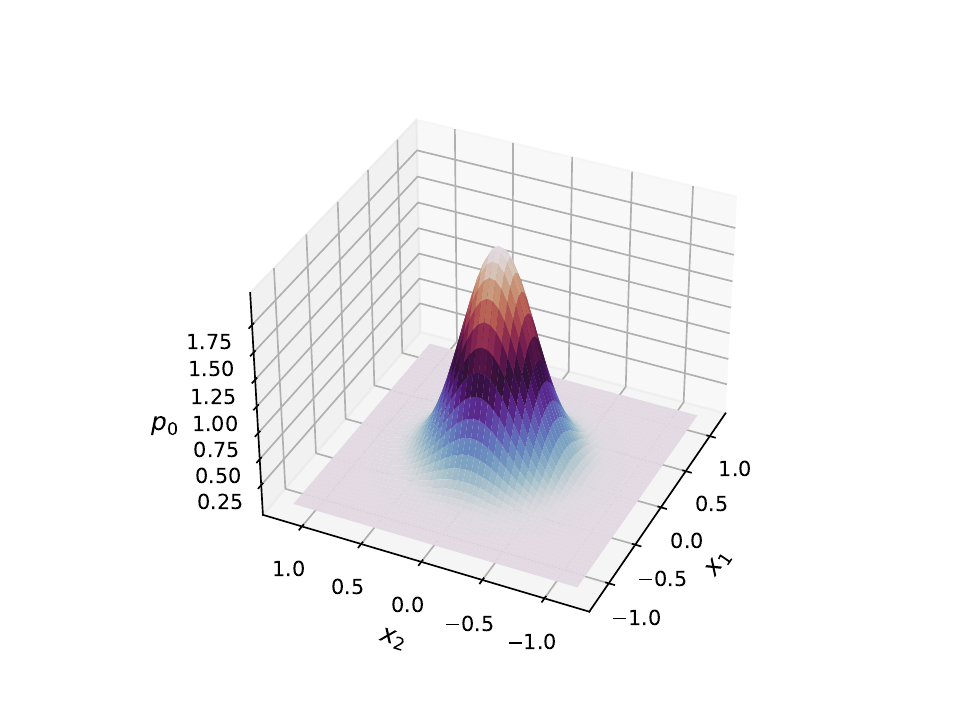}
\captionsetup{labelfont={bf}}
\subcaption{$t=0$, $j=0$}
\end{minipage}
\hspace{.07\linewidth}
\begin{minipage}[t]{.25\textwidth}
\centering
\includegraphics[scale=0.42]{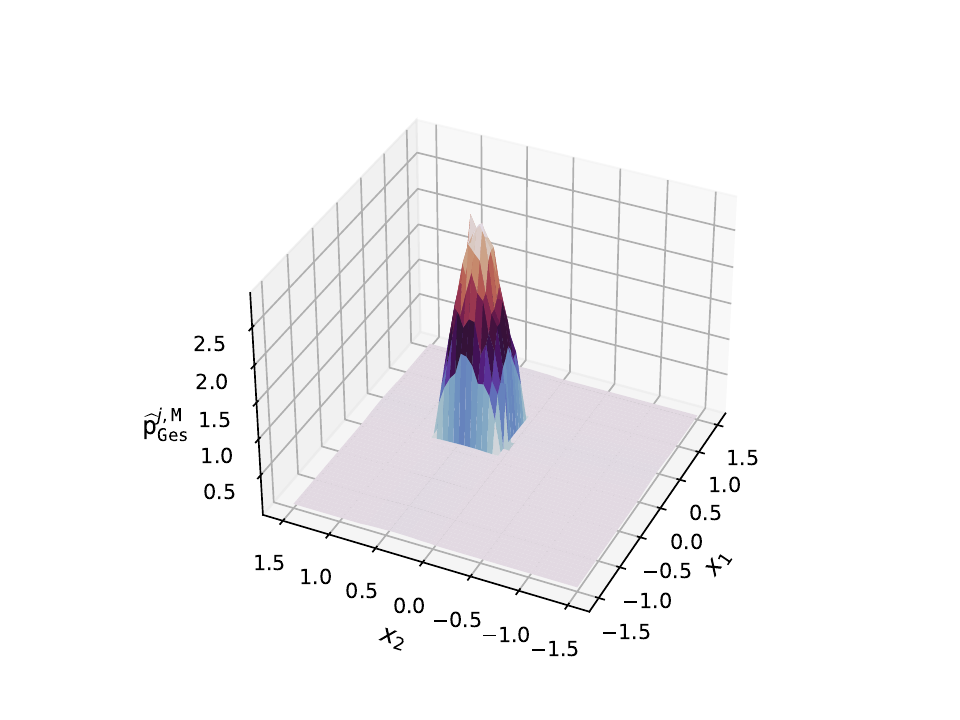}
\captionsetup{labelfont={bf}}
\subcaption{$t=0.5$, $j=50$}
\end{minipage}
\hspace{.07\linewidth}
\begin{minipage}[t]{.25\textwidth}
\centering
\includegraphics[scale=0.42]{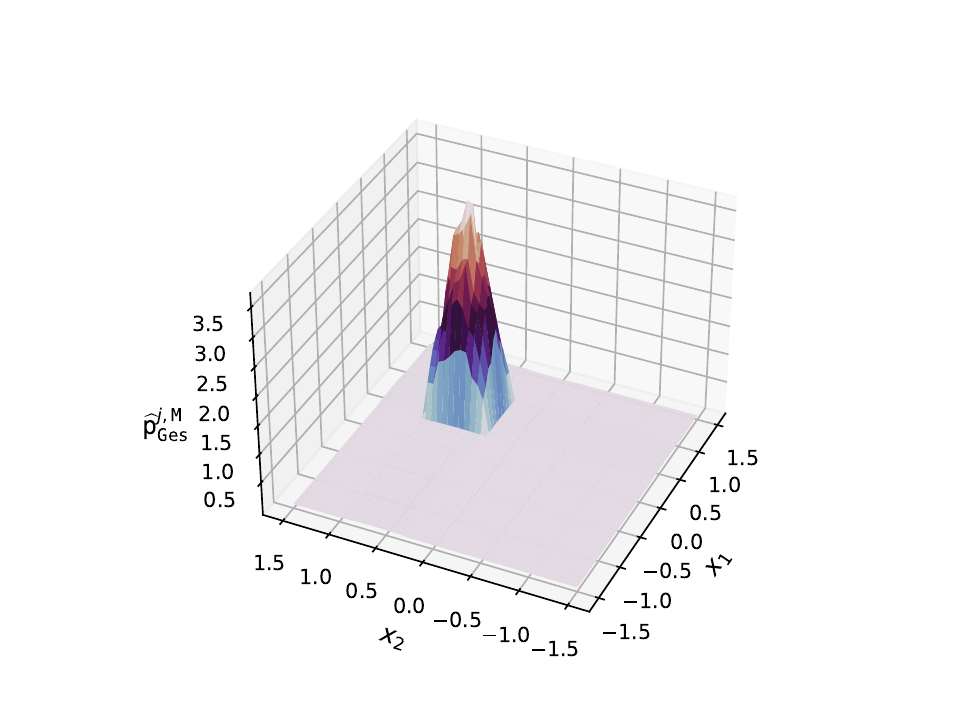}
\captionsetup{labelfont={bf}}
\subcaption{$t=1$, $j=100$}
\end{minipage}

\caption{Example \ref{example1} {\bf b)}: Snapshots for $\widehat{\mathtt{p}}_{\mathtt{Ges}}^{j,\mathtt{M}}$ from Algorithm \ref{strategy} ($\mathtt{M}=2^{16}$, $k_{\mathtt{M}}=2^{8}$).}
\label{figex1_1}

\end{figure}

\end{example}

\medskip

The estimator $\widehat{\mathtt{p}}^{J,\mathtt{M}}_{\mathtt{Ges}}$ is tree-structured, and each node has exactly $\mathtt{N}=\Big\lceil \left(\frac{\mathtt{M}}{k_{\mathtt{M}}}\right)^{\nicefrac{1}{d}} \Big \rceil$ children; see Figure \ref{tree} {\bf (a)}. Starting with $\mathcal{R}_{0,0}^{\mathtt{M}}=\mathbb{R}^{d}$ on \q{level $0$} in the tree in Figure \ref{tree} {\bf (a)}, each subsequent node represents a cell in $\mathbb{R}^{d}$, which is recursively split into $\mathtt{N}$ many equiprobable cells until the last (refinement) level $d$ in the tree is reached. After $d$ Gessaman splits, we obtain a tree of height $d$, where the terminal nodes $\{\mathbf{R}_{r}\,;\, r=0,...,\mathtt{R}_{\mathtt{Ges}}-1\}$ with $\mathtt{R}_{\mathtt{Ges}}:=\mathtt{N}^{d}$ constitute the partition $\pmb{\mathcal{P}}_{\mathtt{Ges}}^{J,\mathtt{M}}$; see Figure \ref{tree} {\bf (a)}. To increase the limited height of this tree, we also propose the alternative BTC-splitting rule \ref{moddunststrategy}, which leads to a binary tree of height $\mathcal{O}\left(\log \mathtt{M}\right)$ that enables a much larger degree of local refinement. Here, the terminal nodes $\{\mathbf{R}_{r}\,;\, r=0,...,\mathtt{R}_{\mathtt{BTC}}-1\}$ with $\mathtt{R}_{\mathtt{BTC}}:=2^{\kappa}=\frac{\mathtt{M}}{k_{\mathtt{M}}}$ constitute the partition $\pmb{\mathcal{P}}_{\mathtt{BTC}}^{J,\mathtt{M}}$; see Figure \ref{tree} {\bf (b)}. This splitting rule in part 3) of Algorithm \ref{strategy} then builds the density estimator $\widehat{\mathtt{p}}^{J,\mathtt{M}}_{\mathtt{BTC}}$ that is detailed in Subsection \ref{3.1.2}.

\medskip

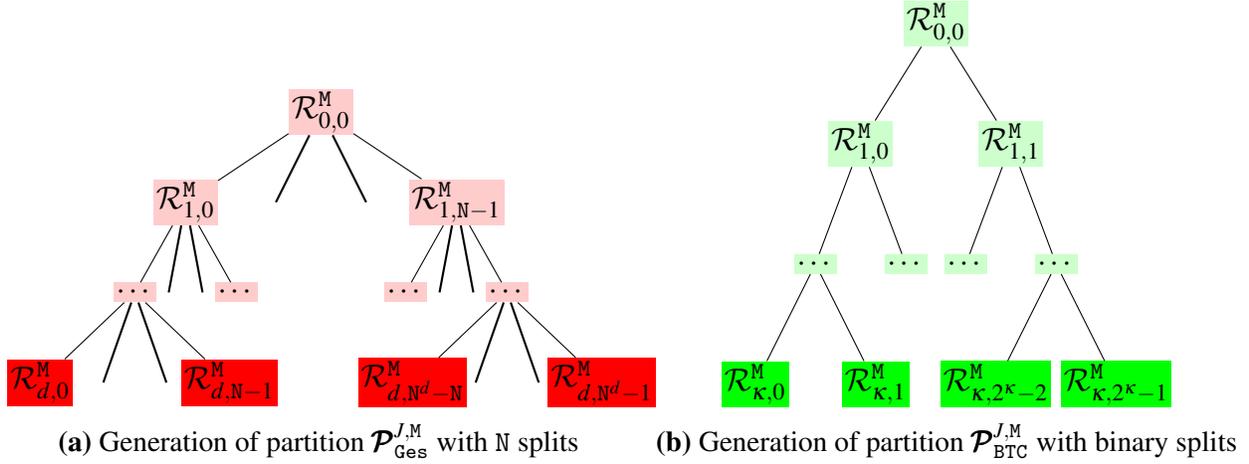
\begin{figure}[h!]
\begin{minipage}[t]{.5\textwidth}
\centering
\begin{tikzpicture}
  [scale=0.3,level distance=40mm,
   every node/.style={fill=red!20,inner sep=1pt},
   level 1/.style={sibling distance=40mm,nodes={fill=red!20}},
   level 2/.style={sibling distance=15mm,nodes={fill=red!20}},
   level 3/.style={sibling distance=28mm,nodes={fill=red!100}}]
  \node {$\mathcal{R}^{\mathtt{M}}_{0,0}$} 
     child {node {$\mathcal{R}^{\mathtt{M}}_{1,0}$ } 
       child { node{$\cdots$} 
         child {node {$\mathcal{R}^{\mathtt{M}}_{d,0}$}}
         child {edge from parent[thick]}
         child{  edge from parent[thick]}
         child {node {$\mathcal{R}^{\mathtt{M}}_{d,\mathtt{N}-1}$}}
       }
       child {edge from parent[thick]
         child[missing]
         child[missing]
       }
       child{  edge from parent[thick]}
       child {node{$\cdots$}}
     }
     child{  edge from parent[thick]}
     child{  edge from parent[thick]}
     child {node {$\mathcal{R}^{\mathtt{M}}_{1,\mathtt{N}-1}$}
     child { node{$\cdots$}} 
     child {edge from parent[thick]}
     child{  edge from parent[thick]}
     child { node{$\cdots$}
      child {node {$\mathcal{R}^{\mathtt{M}}_{d,\mathtt{N}^{d}-\mathtt{N}}$}}
         child {edge from parent[thick]}
         child{  edge from parent[thick]}
         child {node {$\mathcal{R}^{\mathtt{M}}_{d,\mathtt{N}^{d}-1}$}}     
    }
   }   
     ;
\end{tikzpicture}
\captionsetup{labelfont={bf}}
\subcaption{Generation of partition $\pmb{\mathcal{P}}_{\mathtt{Ges}}^{J,\mathtt{M}}$ with $\mathtt{N}$ splits}
\end{minipage}
\begin{minipage}[t]{.5\textwidth}
\centering
\begin{tikzpicture}
  [scale=0.4,level distance=40mm,
   every node/.style={fill=green!20,rectangle,inner sep=1pt},
   level 1/.style={sibling distance=50mm,nodes={fill=green!20}},
   level 2/.style={sibling distance=30mm,nodes={fill=green!20}},
   level 3/.style={sibling distance=40mm,nodes={fill=green!100}}]
  \node {$\mathcal{R}^{\mathtt{M}}_{0,0}$} 
     child {node {$\mathcal{R}^{\mathtt{M}}_{1,0}$ } 
       child { node{$\cdots$} 
         child {node {$\mathcal{R}^{\mathtt{M}}_{\kappa,0}$}}
         child {node {$\mathcal{R}^{\mathtt{M}}_{\kappa,1}$}}
       }
       child {node{$\cdots$}}
     }
     child {node {$\mathcal{R}^{\mathtt{M}}_{1,1}$}
     child { node{$\cdots$}} 
     child { node{$\cdots$}
      child {node {$\mathcal{R}^{\mathtt{M}}_{\kappa,2^{\kappa}-2}$}}
         child {node {$\mathcal{R}^{\mathtt{M}}_{\kappa,2^{\kappa}-1}$}}     
    }
   }   
     ;
\end{tikzpicture}
\captionsetup{labelfont={bf}}
\subcaption {Generation of partition $\pmb{\mathcal{P}}_{\mathtt{BTC}}^{J,\mathtt{M}}$ with binary splits}
\end{minipage}

\caption{$\mathtt{N}-$ary tree for via Gessaman's rule \ref{gessamanstrategy} for $\widehat{\mathtt{p}}^{J,\mathtt{M}}_{\mathtt{Ges}}$ (left) vs. binary tree via BTC-rule \ref{moddunststrategy} for $\widehat{\mathtt{p}}^{J,\mathtt{M}}_{\mathtt{BTC}}$ (right).}
\label{tree}
\end{figure}

In this work, we are interested in complex data $\mathbf{D}_{\mathtt{M}}^{J}$ in the sense of \cite[p.~8]{Brei1}, i.e., data are non-homogeneous and of higher dimension. Here,
\begin{itemize}
\item \q{{\em non-homogeneity}} refers to possibly different relationships between the $d$ variables in different parts of $\mathbb{R}^{d}$, which will be triggered by generators $\mathcal{L}^{\ast}\equiv \mathcal{L}^{\ast}(\mathbf{x})$.
\item The {\em higher the dimensionality}, the sparser and more spread apart the data points are. We approach this difficulty using {\em data-dependent partitions} that foremost resolve places in $\mathbb{R}^{d}$ where data points cluster. Our analysis does not assume that the involved variables are independent --- as is often supposed for related deterministic schemes to perform efficiently. Our strategy exploits \q{{\em dimensionality reduction}} as suggested in \cite[p.~9]{Brei1}; the generated meshing via tree structures fastly segments the data $\mathbf{D}_{\mathtt{M}}^{J}$ into equiprobable cells of the same density, allowing physical interpretation, such as local correlations of the involved $d$ variables.
\end{itemize}

We continue with an example in $d=5$ where the estimators $\widehat{\mathtt{p}}^{J,\mathtt{M}}_{\mathtt{BTC}}$ and $\widehat{\mathtt{p}}^{J,\mathtt{M}}_{\mathtt{Ges}}$ are compared. We observe an increased accuracy of the former, indicating better adapted meshes.

\medskip

\begin{example}
\label{example2}
Let $d=5$ and $T=0.3$. Consider \eqref{eq:fpe}, which corresponds to the associated SDE \eqref{eq:SDE} with
\begin{align*}
\mathbf{b}(\mathbf{x})=\mathbf{0}\,, \qquad \pmb{\sigma}(\mathbf{x})= \frac{\sqrt{2}}{8}\cdot \pmb{Id}_{d}\,,
\end{align*}
where $\mathbf{0}=[0,...,0]^{\top}\in \mathbb{R}^{d}$.
Let $p_{0}$ be the pdf from \eqref{p0ex1} with $\alpha>0$.
The true solution of \eqref{eq:fpe} is given by 
\begin{equation*}
p(t,\mathbf{x})=\frac{1}{\left(2\pi \alpha+\tfrac{\pi}{16} t\right)^{\nicefrac{d}{2}}} \cdot \exp \left(-\frac{16}{32\alpha+t} \cdot \lVert \mathbf{x} \rVert_{\mathbb{R}^{d}}^{2} \right)\qquad (t,\mathbf{x})\in [0,T]\times \mathbb{R}^{d}\,.
\end{equation*}
{\bf a)} For $\tau=\frac{1}{100}$, we consider different choices of $\alpha>0$ in the representation of the initial profile $p_{0}$ in \eqref{p0ex1}. The parameter $\alpha$ again quantifies gradients in $p_{0}$. The comparative plots in Figure \ref{figex13_alphaError} indicate that $\widehat{\mathtt{p}}^{J,\mathtt{M}}_{\mathtt{BTC}}$ reaches a given accuracy with significantly fewer samples than $\widehat{\mathtt{p}}^{J,\mathtt{M}}_{\mathtt{Ges}}$. As we see in Figure \ref{figex13_alphaError} {\bf (a)}, the size $\mathtt{M}$ has to grow when gradients steepen in $p_{0}$ (i.e., $\alpha$ shrinks). Furthermore an initial profile $p_{0}$ with smaller gradients also has a positive impact on the overall performance of Algorithm \ref{strategy}: for a fixed choice of $\mathtt{M}=2^{24}$ and $k_{\mathtt{M}}=2^{9}$, the $\mathbb{L}^{\infty}-$error for both $\widehat{\mathtt{p}}^{J,\mathtt{M}}_{\mathtt{Ges}}$ and $\widehat{\mathtt{p}}^{J,\mathtt{M}}_{\mathtt{BTC}}$ is decreasing for increasing $\alpha$; see Figure \ref{figex13_alphaError} {\bf (b)}.

\begin{figure}[h!]
\begin{minipage}[t]{.4\textwidth}
\centering
\includegraphics[scale=0.5]{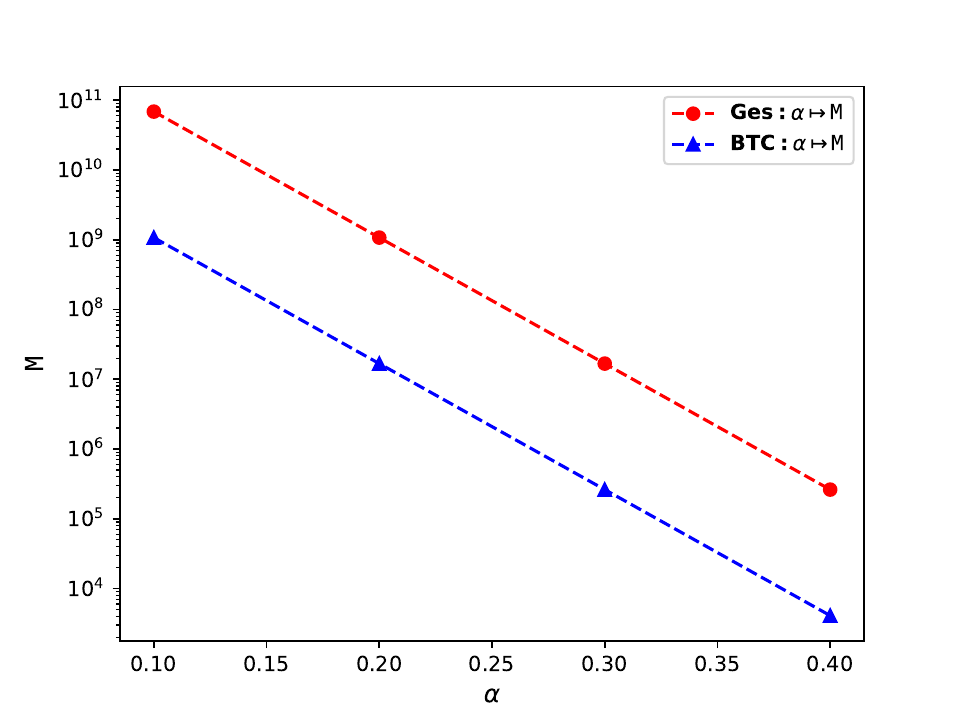}
\captionsetup{labelfont={bf}}
\subcaption{$\alpha\mapsto \mathtt{M}$}
\end{minipage}
\hspace{.1\linewidth}
\begin{minipage}[t]{.4\textwidth}
\centering
\includegraphics[scale=0.5]{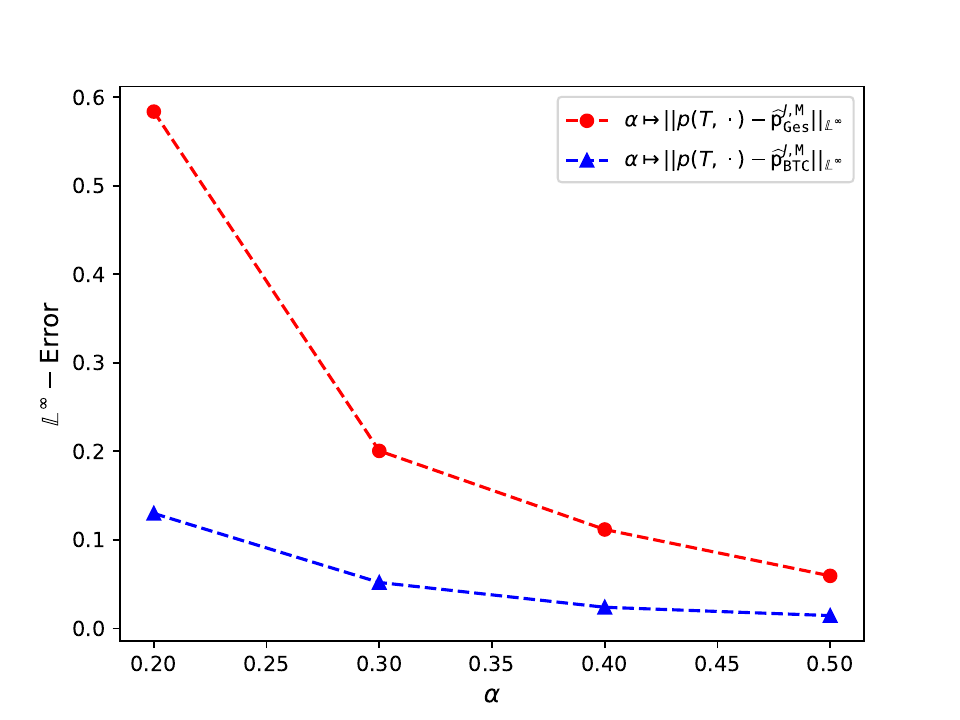}
\captionsetup{labelfont={bf}}
\subcaption{$\alpha\mapsto \lVert p(T,\cdot)-\widehat{\mathtt{p}}^{J,\mathtt{M}} \rVert_{\mathbb{L}^{\infty}}$}
\end{minipage}

\caption{Example \ref{example2} {\bf a)}: {\bf (a)} Minimum $\mathtt{M}$ for different $\alpha$ to have $\mathbb{L}^{\infty}-$error below $0.1$ (with a fixed ratio of $\mathtt{M}=2^{2N}$ and $k_{\mathtt{M}}=2^{N-2}$, $(N\in \mathbb{N})$). {\bf (b)} $\mathbb{L}^{\infty}-$error plots for $\widehat{\mathtt{p}}^{J,\mathtt{M}}\equiv \widehat{\mathtt{p}}^{J,\mathtt{M}}_{\mathtt{Ges}}$ resp.~$\widehat{\mathtt{p}}^{J,\mathtt{M}}\equiv \widehat{\mathtt{p}}^{J,\mathtt{M}}_{\mathtt{BTC}}$ and different values of $\alpha$ ($\mathtt{M}=2^{24}$, $k_{\mathtt{M}}=2^{9}$).}

\label{figex13_alphaError}

\end{figure}

{\bf b)} Let $\alpha=\frac{1}{2\pi}$. For $\tau=\frac{1}{100}$, $\mathtt{M}=2^{26}$ and $k_{\mathtt{M}}=2^{11}$, Figure \ref{partitions5d} shows the partitions $\pmb{\mathcal{P}}_{\mathtt{Ges}}^{J,\mathtt{M}}$ resp.~$\pmb{\mathcal{P}}_{\mathtt{BTC}}^{J,\mathtt{M}}$ of $\mathbb{R}^{5}$ restricted to the plane $\mathbf{P_{1,2}}:=\{(x_{1},x_{2})\,:\, x_{1},x_{2}\in \mathbb{R}\}$. Figure \ref{solutions5d} shows the related approximated solutions $\restr{\widehat{\mathtt{p}}^{J,\mathtt{M}}_{\mathtt{Ges}}}{\mathbf{P_{1,2}}}$ resp.~$\restr{\widehat{\mathtt{p}}^{J,\mathtt{M}}_{\mathtt{BTC}}}{\mathbf{P_{1,2}}}$, and Figure \ref{solutionszoom5d} shows a zoomed view of them from above: while $\restr{\pmb{\mathcal{P}}^{J,\mathtt{M}}_{\mathtt{Ges}}}{ \mathbf{P_{1,2}}}$ in Figure \ref{partitions5d} {\bf (a)} looks like a \q{nested tensorised mesh}, the second choice $\restr{\pmb{\mathcal{P}}^{J,\mathtt{M}}_{\mathtt{BTC}}}{ \mathbf{P_{1,2}}}$ in Figure \ref{partitions5d} {\bf (b)} detects curved features of the solution; see Figure \ref{solutionszoom5d}. This increased resolution leads to smaller errors for the estimator $\widehat{\mathtt{p}}^{J,\mathtt{M}}\equiv \widehat{\mathtt{p}}^{J,\mathtt{M}}_{\mathtt{BTC}}$ in Algorithm \ref{strategy}; see also Figure \ref{figex13_alphaError}.

\begin{figure}[h!]
\begin{minipage}[t]{.4\textwidth}
\centering
\includegraphics[scale=0.5]{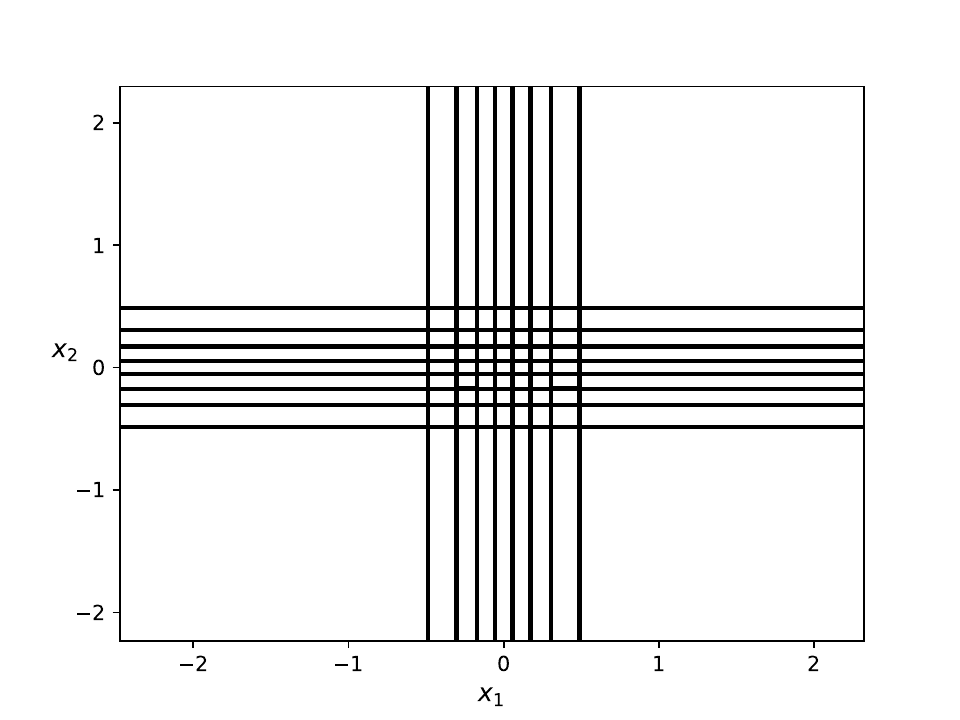}
\captionsetup{labelfont={bf}}
\subcaption{$\restr{\pmb{\mathcal{P}}^{J,\mathtt{M}}_{\mathtt{Ges}}}{ \mathbf{P_{1,2}}}$}
\end{minipage}
\hspace{.1\linewidth}
\begin{minipage}[t]{.4\textwidth}
\centering
\includegraphics[scale=0.5]{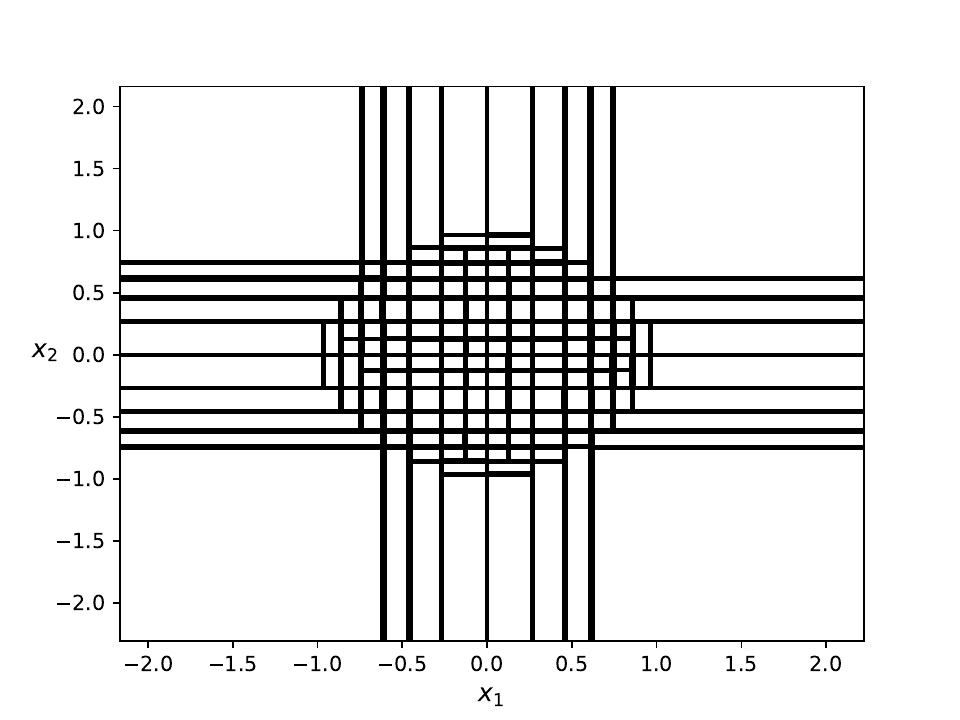}
\captionsetup{labelfont={bf}}
\subcaption{$\restr{\pmb{\mathcal{P}}^{J,\mathtt{M}}_{\mathtt{BTC}}}{ \mathbf{P_{1,2}}}$}
\end{minipage}

\caption{Example \ref{example2} {\bf b)}: Partition $\pmb{\mathcal{P}}^{J,\mathtt{M}}_{\mathtt{Ges}}$ resp.~$\pmb{\mathcal{P}}^{J,\mathtt{M}}_{\mathtt{BTC}}$ generated via {\bf (a)} Gessaman's rule \ref{gessamanstrategy} resp.~{\bf (b)} BTC-rule \ref{moddunststrategy} restricted to $\mathbf{P_{1,2}}$.}
\label{partitions5d}

\end{figure}

\begin{figure}[h!]
\begin{minipage}[t]{.4\textwidth}
\centering
\includegraphics[scale=0.6]{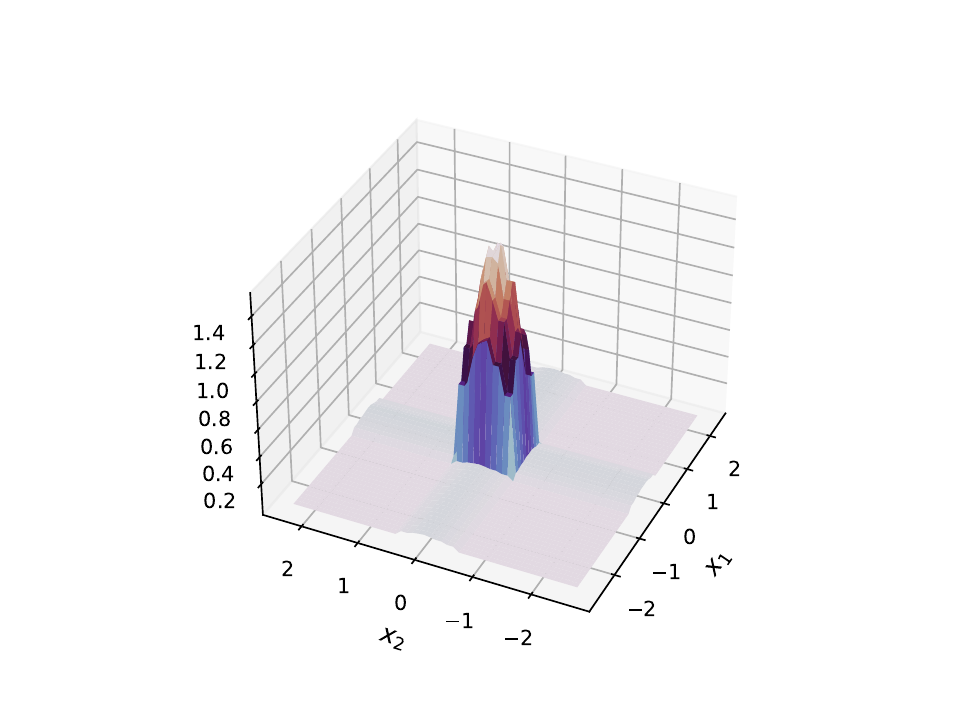}
\captionsetup{labelfont={bf}}
\subcaption{$\restr{\widehat{\mathtt{p}}^{J,\mathtt{M}}_{\mathtt{Ges}}}{\mathbf{P_{1,2}}}$}
\end{minipage}
\hspace{.08\linewidth}
\begin{minipage}[t]{.4\textwidth}
\centering
\includegraphics[scale=0.6]{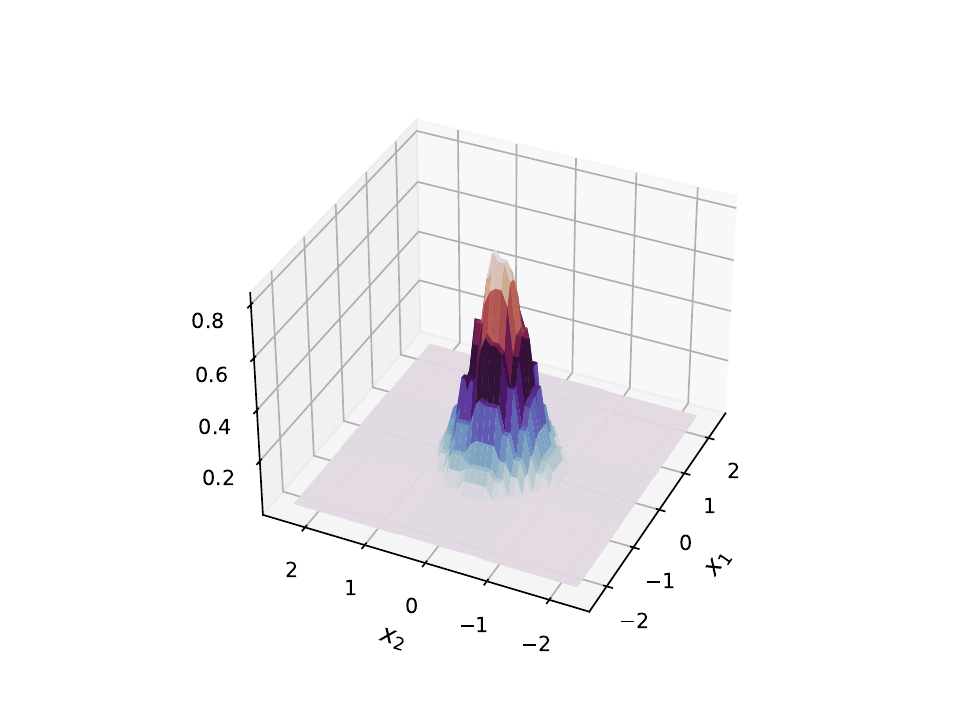}
\captionsetup{labelfont={bf}}
\subcaption{$\restr{\widehat{\mathtt{p}}^{J,\mathtt{M}}_{\mathtt{BTC}}}{\mathbf{P_{1,2}}}$}
\end{minipage}

\caption{Example \ref{example2} {\bf b)}: Approximated solutions {\bf (a)} $\widehat{\mathtt{p}}^{J,\mathtt{M}}_{\mathtt{Ges}}$ resp.~{\bf (b)} $\widehat{\mathtt{p}}^{J,\mathtt{M}}_{\mathtt{BTC}}$ restricted to $\mathbf{P_{1,2}}$.}
\label{solutions5d}

\end{figure}

\begin{figure}[h!]

\begin{minipage}[t]{.25\textwidth}
\centering
\includegraphics[scale=0.38]{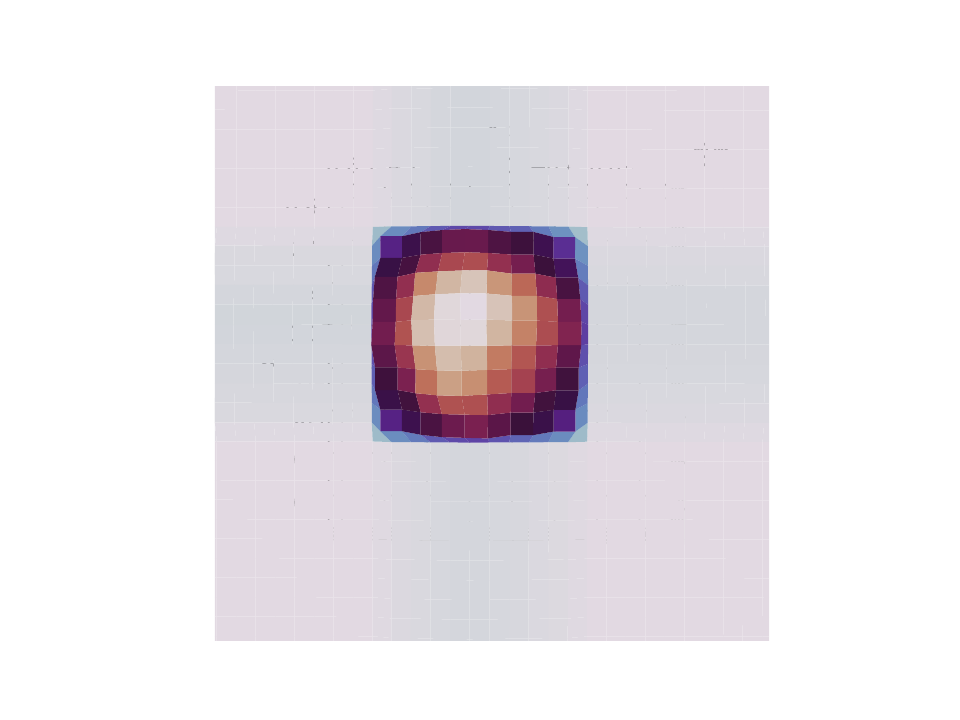}
\captionsetup{labelfont={bf}}
\subcaption{$\restr{\widehat{\mathtt{p}}^{J,\mathtt{M}}_{\mathtt{Ges}}}{\mathbf{P_{1,2}}}$}
\end{minipage}
\hspace{.05\linewidth}
\begin{minipage}[t]{.25\textwidth}
\centering
\includegraphics[scale=0.38]{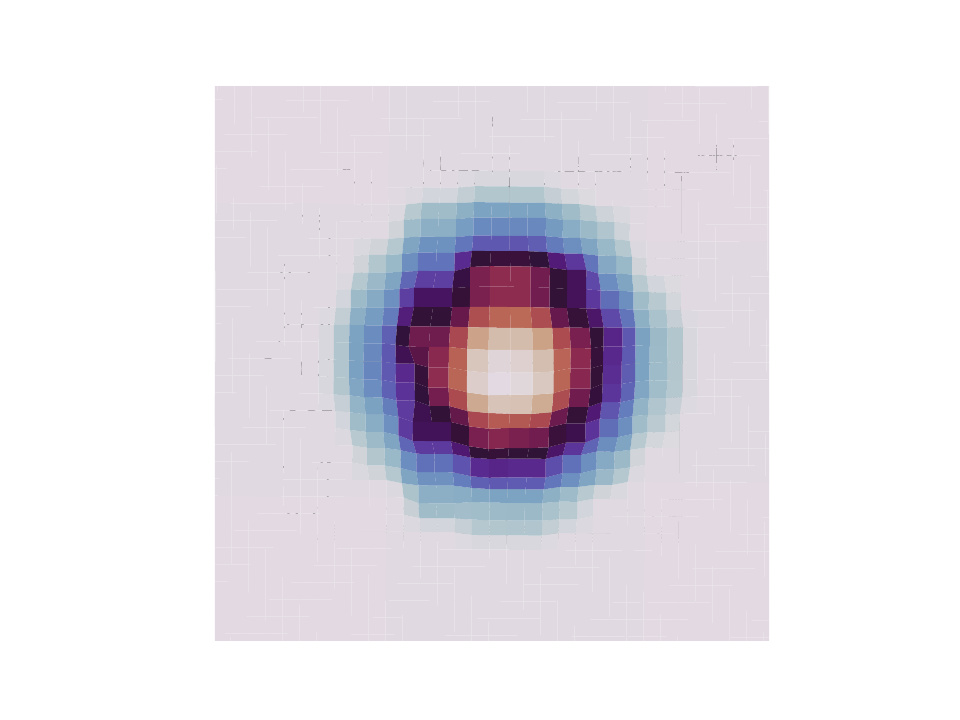}
\captionsetup{labelfont={bf}}
\subcaption{$\restr{\widehat{\mathtt{p}}^{J,\mathtt{M}}_{\mathtt{BTC}}}{\mathbf{P_{1,2}}}$}
\end{minipage}
\hspace{.05\linewidth}
\begin{minipage}[t]{.25\textwidth}
\centering
\includegraphics[scale=0.38]{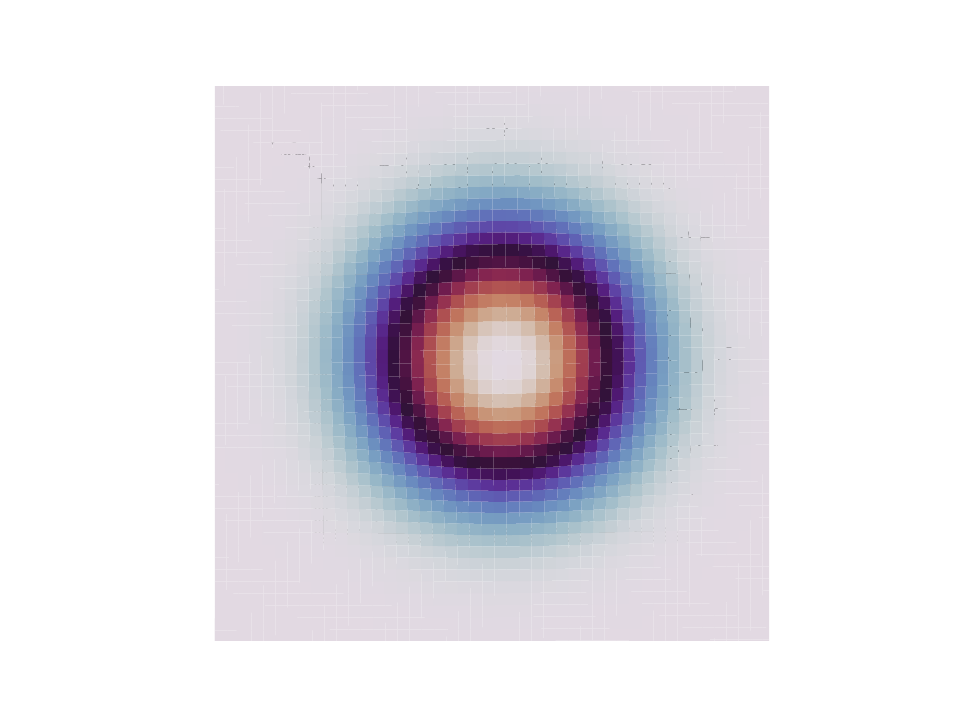}
\captionsetup{labelfont={bf}}
\subcaption{$\restr{p(T,\cdot)}{\mathbf{P_{1,2}}}$}
\end{minipage}

\caption{Example \ref{example2} {\bf b)}: View from above: {\bf (a)} $\widehat{\mathtt{p}}^{J,\mathtt{M}}_{\mathtt{Ges}}$ resp.~{\bf (b)} $\widehat{\mathtt{p}}^{J,\mathtt{M}}_{\mathtt{BTC}}$ resp.~{\bf (c)} $p(T,\cdot)$ restricted to $\mathbf{P_{1,2}}$.}
\label{solutionszoom5d}

\end{figure}

\end{example}

\

The next example discusses the performance of the estimator $\widehat{\mathtt{p}}^{J,\mathtt{M}}_{\mathtt{BTC}}$ in a setting when $\mathcal{L}^{\ast}(\mathbf{x})$ given in \eqref{eq:operator} depends on $\mathbf{x}\in \mathbb{R}^{d}$, and where $d=8$.

\begin{example}
\label{example3}
Let $d=8$ and $T=0.2$. Consider \eqref{eq:fpe}, which corresponds to the associated SDE \eqref{eq:SDE} with
\begin{align*}
\mathbf{b}(\mathbf{x})=[0,0,2,2,2,x_{6},x_{7},x_{8}]^{\top}\,, \qquad \pmb{\sigma}(\mathbf{x})=\frac{1+\lVert \mathbf{x} \rVert_{\mathbb{R}^{d}}}{10}\cdot \pmb{Id}_{d}\,.
\end{align*}
Let $p_{0}$ be the pdf from \eqref{p0ex1} with $\alpha=\frac{1}{2\pi}$. For a fixed choice of $\tau=0.01$, $\mathtt{M}=2^{32}$ and $k_{\mathtt{M}}=2^{10}$, Figure \ref{ex3dynamics} shows approximated solution profiles $\widehat{\mathtt{p}}^{J,\mathtt{M}}_{\mathtt{BTC}}$ when restricted to different coordinate axes. In the style of the notation in Example \ref{example2}, we denote the $x_{i}-$axis by $\mathbf{P_{i}}:=\{x_{i}\,:x_{i} \in \mathbb{R}\}$, $i=1,...,8$. We observe a different structure of $\widehat{\mathtt{p}}^{J,\mathtt{M}}_{\mathtt{BTC}}$ along each coordinate axis mainly due to $\mathbf{b}(\mathbf{x})$, but also $\pmb{\sigma}(\mathbf{x})$: while Figure \ref{ex3dynamics} {\bf (a)} illustrates the smoothing/flattening dynamics of $\restr{\widehat{\mathtt{p}}^{J,\mathtt{M}}_{\mathtt{BTC}}}{\mathbf{P_{1}}}$ (note that $\restr{\mathbf{b}(\mathbf{x})}{\mathbf{P_{1}}}\equiv 0$ in this case), Figure \ref{ex3dynamics} {\bf (b)} shows a convection of the approximated solution in the direction given by $\restr{\mathbf{b}(\mathbf{x})}{\mathbf{P_{4}}}\equiv 2$. Figure \ref{ex3dynamics} {\bf (c)} shows a snapshot of $\restr{\widehat{\mathtt{p}}^{J,\mathtt{M}}_{\mathtt{BTC}}}{\mathbf{P_{8,7}}}$, {\em i.e.}, of $\widehat{\mathtt{p}}^{J,\mathtt{M}}_{\mathtt{BTC}}$ restricted to the $x_{8}-x_{7}-$plane.


\begin{figure}[h!]
\begin{minipage}[t]{.25\textwidth}
\centering
\includegraphics[scale=0.2]{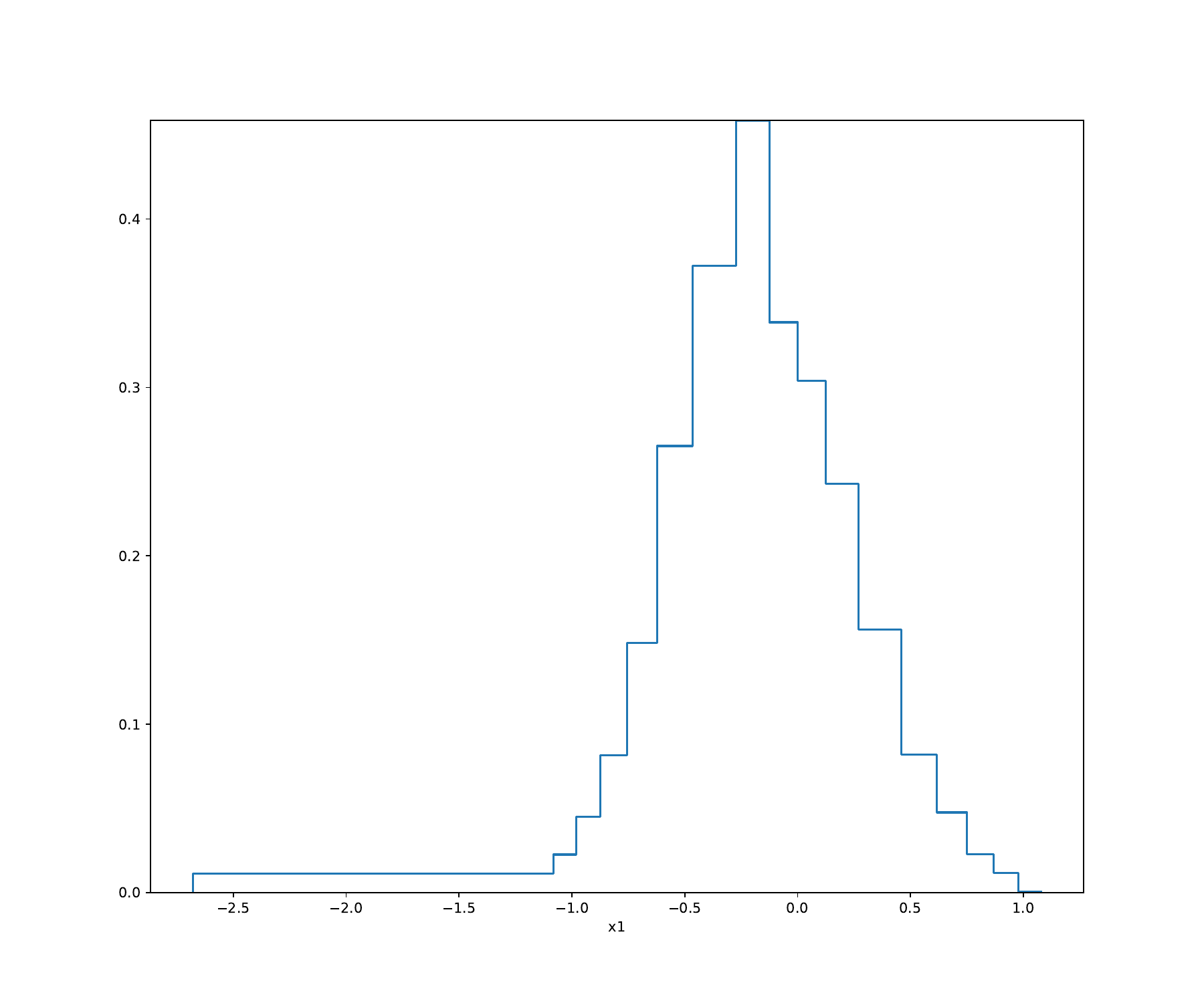}
\captionsetup{labelfont={bf}}
\subcaption{$\restr{\widehat{\mathtt{p}}^{J,\mathtt{M}}_{\mathtt{BTC}}}{\mathbf{P_{1}}}$}
\end{minipage}
\hspace{.08\linewidth}
\begin{minipage}[t]{.25\textwidth}
\centering
\includegraphics[scale=0.2]{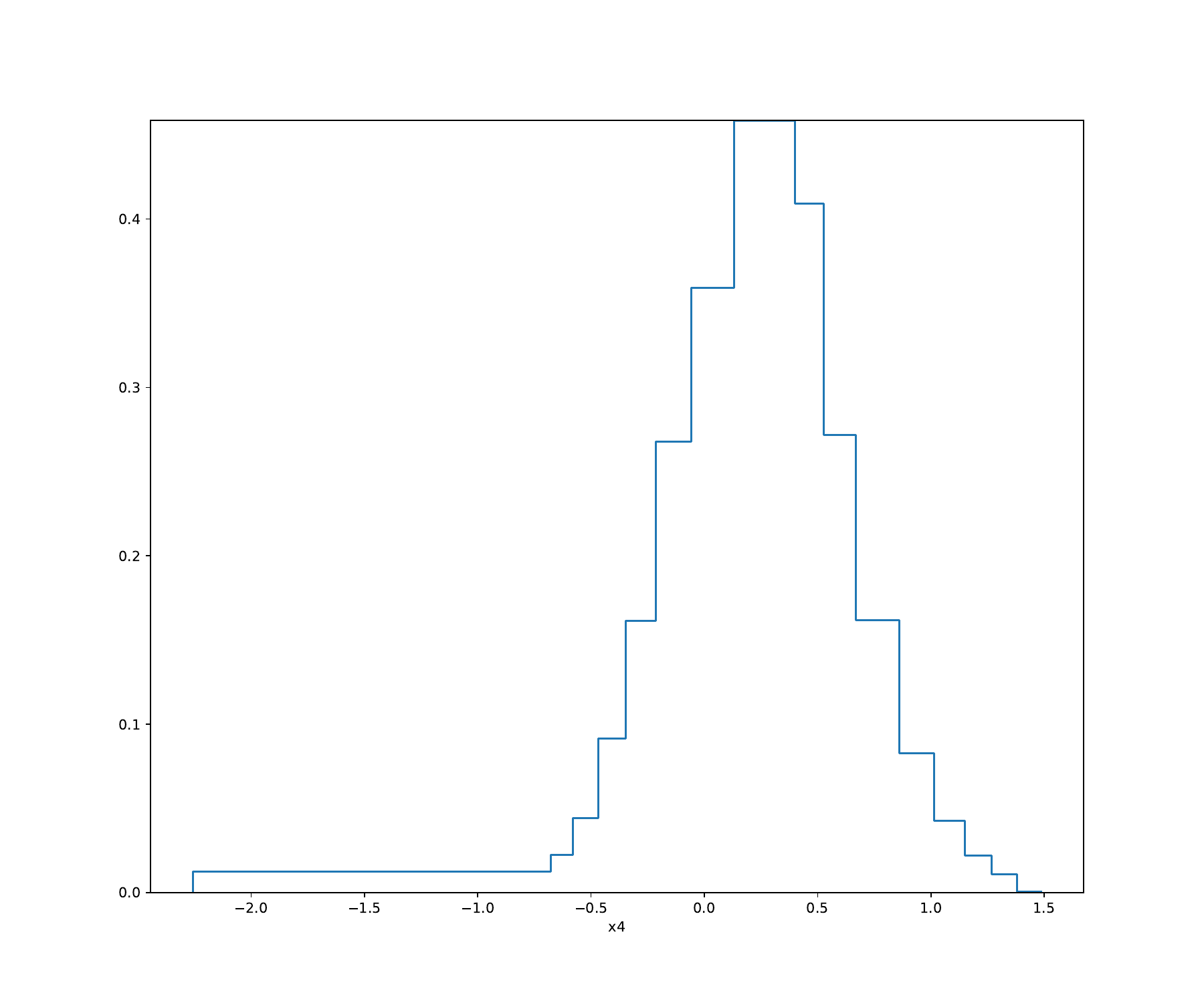}
\captionsetup{labelfont={bf}}
\subcaption{$\restr{\widehat{\mathtt{p}}^{J,\mathtt{M}}_{\mathtt{BTC}}}{\mathbf{P_{4}}}$}
\end{minipage}
\hspace{.08\linewidth}
\begin{minipage}[t]{.25\textwidth}
\centering
\includegraphics[scale=0.2]{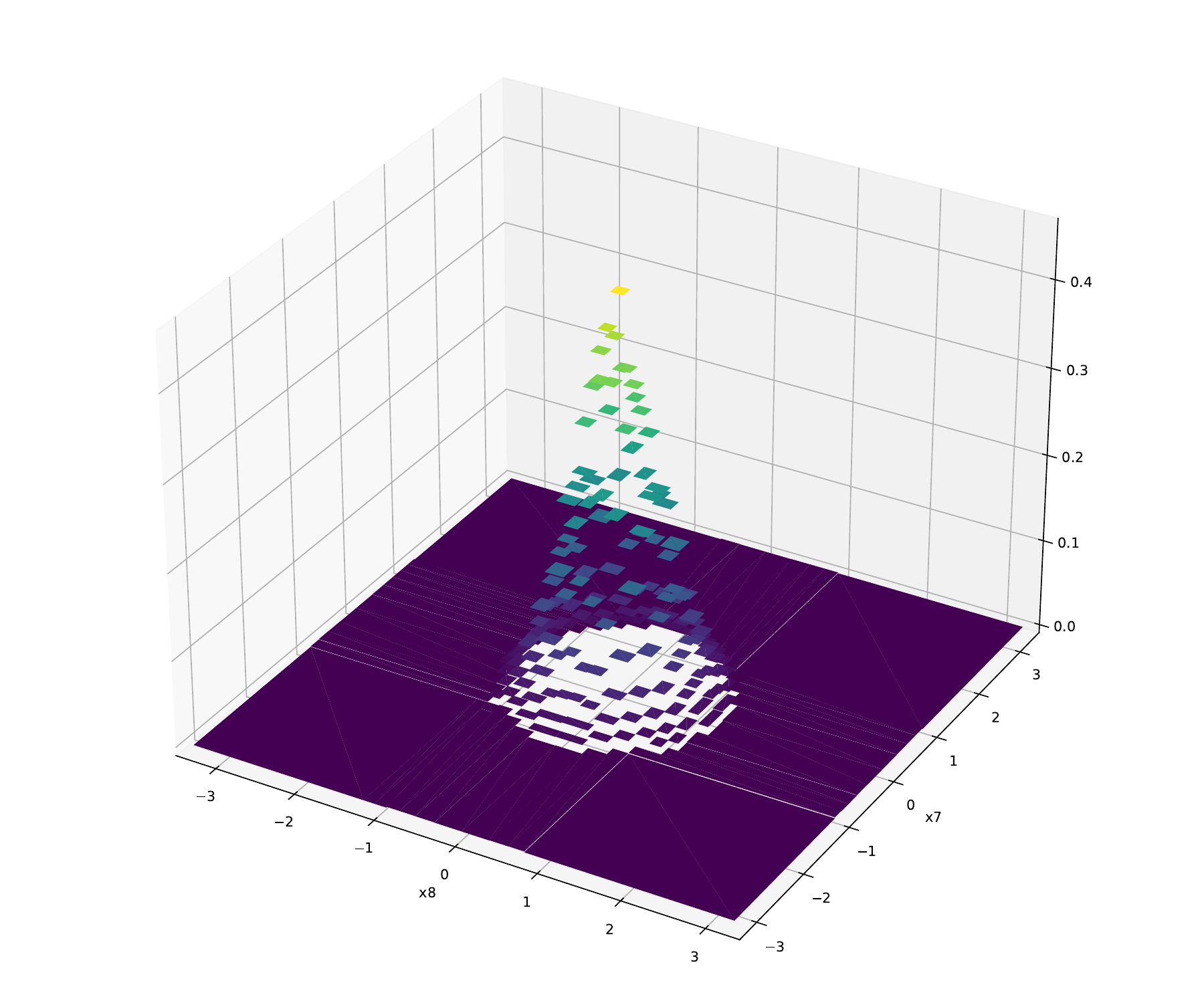}
\captionsetup{labelfont={bf}}
\subcaption{$\restr{\widehat{\mathtt{p}}^{J,\mathtt{M}}_{\mathtt{BTC}}}{\mathbf{P_{8,7}}}$}
\end{minipage}

\caption{ Example \ref{example3}: Approximated solution profile restricted to different coordinate axes ($\mathtt{M}=2^{32}$, $k_{\mathtt{M}}=2^{10}$).}
\label{ex3dynamics}

\end{figure}

\end{example}

\medskip

As mentioned before, our main theoretical result is to verify \eqref{eq:convergence} for the density estimator $\widehat{\mathtt{p}}_{\mathtt{Ges}}^{J,\mathtt{M}}$ in Section \ref{3}: for this purpose, we first add and subtract the pdf $p^{J}$ of the random variable $\mathbf{Y}^{J}$, and use the triangle inequality to obtain
\begin{equation}
\label{eq:convergencestep1}
\lVert p(T,\cdot)-\widehat{\mathtt{p}}_{\mathtt{Ges}}^{J,\mathtt{M}}\rVert_{\mathbb{L}^{1}}\leq \underbrace{\lVert p(T,\cdot)-p^{J}\rVert_{\mathbb{L}^{1}}}_{=:{\bf I}} + \underbrace{\lVert p^{J}-\widehat{\mathtt{p}}_{\mathtt{Ges}}^{J,\mathtt{M}}\rVert_{\mathbb{L}^{1}}}_{=:{\bf II}} \qquad \mathbb{P}-a.s.
\end{equation}
Term ${\bf I}$ in \eqref{eq:convergencestep1} measures the error between the pdfs of $\mathbf{X}_{T}$ and $\mathbf{Y}^{J}$, which tends to zero when the time step size $\tau$ tends to zero; for its convergence proof ({\em cf.}~Theorem \ref{modballythm}), we benefit from ideas and results in \cite{bally}. On the other hand, term ${\bf II}$ measures the error between the pdf $p^{J}$ and its empirical (statistical) counterpart $\widehat{\mathtt{p}}_{\mathtt{Ges}}^{J,\mathtt{M}}$ in \eqref{eq:estimator} --- which is random --- and $\mathbb{P}-a.s.$ converges to zero when $\mathtt{M}$ tends to infinity. The convergence of ${\bf II}$ ({\em cf.}~Theorem \ref{thmgessamanconvergence}) builds on the concept of strong $\mathbb{L}^{1}-$consistency of histogram-based density estimators as introduced in \cite{lugosi}. Furthermore we recall from \cite{lugosi} general conditions for a density estimator based on a data-dependent partitioning strategy to be strongly $\mathbb{L}^{1}-$consistent; see Theorem \ref{thmlugosi}.
The proof of Theorem \ref{thmgessamanconvergence} for $\widehat{\mathtt{p}}^{J,\mathtt{M}}\equiv \widehat{\mathtt{p}}^{J,\mathtt{M}}_{\mathtt{Ges}} $ in ${\bf II}$ then consists of verifying the general conditions in Theorem \ref{thmlugosi} for the special estimator $\widehat{\mathtt{p}}^{J,\mathtt{M}}_{\mathtt{Ges}}$ that uses statistically equivalent cells --- ensuring that \eqref{eq:convergence} holds. 

\medskip

To conclude the introduction, we summarise the advantages of the new method to obtain {\em data-dependent} density estimators, which include broad usability to solve \eqref{eq:fpe} for non-constant $\mathcal{L}^{\ast}\equiv \mathcal{L}^{\ast}(\mathbf{x})$ in higher dimensions, fast implementation, physical interpretability due to the tree-structured construction (see also \cite[p.~55]{Brei1}), mesh adaptivity, and a general convergence theory {\em without} strong assumptions on data or solutions. A disadvantage of Algorithm \ref{strategy}, however, is the growth of the sample size $\mathtt{M}$ with the dimension $d$ --- a problem that we plan to partially compensate in future works by adjusting Algorithm \ref{strategy}: Promising tools here include the use of pruning optimal trees \cite[Ch.~3]{Brei1} to reduce complexity, and cross-validation \cite[Ch.'s 7,8]{gyoerfi2} as well as random forests to stabilise density estimators; moreover, bootstrap methods (\q{bagging}; see \cite{B1}) may help to generate replicate samplings efficiently. Conceptually, the need for larger $\mathtt{M}-$samples again reflects the \q{curse of dimensionality}, which is inherent to problem \eqref{eq:fpe} when we consider general non-structured solutions of \eqref{eq:fpe}, \eqref{eq:operator}, which may only be reduced when additional assumptions are made on data and solutions.\\
The remainder of this work is organised as follows: Section \ref{2} settles approximability (with rates) of $p(T,\cdot)$ by $p^{J}$. Section \ref{3} shows consistency of the density estimator $\widehat{\mathtt{p}}^{J,\mathtt{M}}_{\mathtt{Ges}}$ from Algorithm \ref{strategy}, where Gessaman's rule \ref{gessamanstrategy} is used in step 3) of the algorithm; moreover, we introduce the binary-tree-based estimator $\widehat{\mathtt{p}}^{J,\mathtt{M}}_{\mathtt{BTC}}$ with BTC-rule \ref{moddunststrategy} as an alternative choice in step 3).
Extended computational experiments in Section \ref{4} address practical choices for the parameter $k_{\mathtt{M}}$ to set sample sizes of constructed cells for both density estimators.

\section{Assumptions and the law of the Euler scheme}
\label{2}

Subsection \ref{2.1} lists basic requirements on data $\mathbf{b}$, $\pmb{\sigma}$ and $p_{0}$ in \eqref{eq:fpe} resp.~\eqref{eq:SDE}, which guarantee the existence of the pdfs $p(T,\cdot)$ resp.~$p^{J}$ of $\mathbf{X}_{T}$ resp.~$\mathbf{Y}^{J}$ from \eqref{eq:SDE} resp.~\eqref{eq:euler}. The results in Subsection \ref{2.2} make sure that $p^{J}$ converges.

\subsection{Assumptions on data of the SDE \eqref{eq:SDE}}
\label{2.1}

Throughout this work, $\bigl(\Omega,\mathcal{F},\{\mathcal{F}_{t}\}_{t\geq 0},\mathbb{P}\bigr)$ is a given filtered probability space with natural filtration generated by the Wiener process $\mathbf{W}$ in \eqref{eq:SDE}. For $1\leq q < \infty$, and a given integrable function $f:\mathbb{R}^{d}\to \mathbb{R}$, we denote by $\lVert f \rVert_{\mathbb{L}^{q}}:=\Big(\int\limits_{\mathbb{R}^{d}}\,\lvert f(\mathbf{x})\rvert^{q}\,\mathrm{d}\mathbf{x}\Big)^{\nicefrac{1}{q}}$ the $\mathbb{L}^{q}-$norm of $f$, and by $\lVert f \rVert_{\mathbb{L}^{\infty}}:=\esssup\limits_{\mathbf{x}\in \mathbb{R}^{d}}\lvert f(\mathbf{x})\rvert$ the $\mathbb{L}^{\infty}-$norm of $f$. For two sets $\pmb{\mathcal{A}}$ and $\pmb{\mathcal{B}}$, we denote by $\mathcal{C}^{\infty}(\pmb{\mathcal{A}};\pmb{\mathcal{B}})$ the space of all smooth functions $f:\pmb{\mathcal{A}}\to \pmb{\mathcal{B}}$. If $\pmb{\mathcal{B}}=\mathbb{R}$, we write $\mathcal{C}^{\infty}(\pmb{\mathcal{A}})$. Furthermore, we denote by $\mathcal{C}^{\infty}_{0}(\mathbb{R}^{d})$ the space of all smooth functions $f:\mathbb{R}^{d}\to \mathbb{R}$ vanishing at infinity. Moreover, we denote by $\lVert \cdot \rVert_{\mathbb{R}^{d}}$ the Euclidean norm, and by $\langle \cdot, \cdot \rangle_{\mathbb{R}^{d}}$ we denote the corresponding scalar product in $\mathbb{R}^{d}$. In addition, we denote by $\lVert \cdot \rVert_{\mathbb{R}^{d\times d}}$ the spectral matrix norm. Throughout this work, we assume for simplicity:

\begin{itemize}
\item[{\bf (A1)}] $\mathbf{b}\in \mathcal{C}^{\infty}(\mathbb{R}^{d};\mathbb{R}^{d})$, with bounded derivatives of any order.
\item[{\bf (A2)}] $\pmb{\sigma}\in \mathcal{C}^{\infty}(\mathbb{R}^{d};\mathbb{R}^{d\times d})$, with bounded derivatives, and moreover, $\pmb{\sigma}(\cdot)\pmb{\sigma}^{\top}(\cdot)$ satisfies the {\em uniform ellipticity condition}, i.e., there exists a constant $\theta>0$, such that
\begin{equation*}
\langle \mathbf{y},\pmb{\sigma}(\mathbf{x})\pmb{\sigma}^{\top}(\mathbf{x})\mathbf{y}\rangle_{\mathbb{R}^{d}}\geq \theta \cdot \lVert \mathbf{y} \rVert_{\mathbb{R}^{d}}^{2}\qquad \text{for all } \mathbf{x},\mathbf{y}\in \mathbb{R}^{d}\,.
\end{equation*}
\item[{\bf (A3)}] $p_{0}:\mathbb{R}^{d}\to [0,\infty)$ is a pdf; in particular $\lVert p_{0}\rVert_{\mathbb{L}^{1}}=1$.
\end{itemize}

\medskip

It is well-known that by the assumptions above, there exists a unique classical solution $p:[0, T]\times \mathbb{R}^{d}\to [0,\infty)$ with $p(t,\cdot) \in \mathcal{C}^{\infty}_{0}(\mathbb{R}^{d})$ for all $t\in (0, T]$ of \eqref{eq:fpe}, which is positive for every fixed time $t\in (0, T]$, i.e., $p(t,\cdot)>0$ on $\mathbb{R}^{d}$; see e.g.~\cite[Thm.~8.12.1]{Krylov}.

\subsection{Convergence and properties of the density $p^{J}$ of $\mathbf{Y}^{J}$ from \eqref{eq:euler}}
\label{2.2}

Assumptions {\bf (A1)} -- {\bf (A3)} in Subsection \ref{2.1} also guarantee the existence of the pdf $p^{J}$ for the random variable $\mathbf{Y}^{J}$ from \eqref{eq:euler}; see e.g.~\cite{bally}.
In this subsection, we prove the $\mathbb{P}-a.s.$ convergence of $p^{J}$ to $p(T,\cdot)$ in $\mathbb{L}^{q}$ ($1\leq q\leq \infty$) by extending \cite[Corollary 2.7]{bally} to include random initial conditions; see Theorem Theorem \ref{modballythm} below.
First, we introduce some notation: for $\mathbf{y}\in \mathbb{R}^{d}$, we denote by $(t,\mathbf{x})\mapsto p_{\mathbf{y}}(t,\mathbf{x})$ the pdf of the associated SDE process $\mathbf{X}^{\mathbf{y}}\equiv \{\mathbf{X}_{t}^{\mathbf{y}}\,;\,t\in [0,T]\}$, which solves \eqref{eq:SDE} with $\mathbf{X}_{0}=\mathbf{y}$. Similarly, we denote by $\mathbf{x}\mapsto p_{\mathbf{y}}^{J}(\mathbf{x})$ the pdf associated with the Euler iterate $\mathbf{Y}^{J}$ from \eqref{eq:euler}, where $\mathbf{Y}^{0}=\mathbf{y}\in \mathbb{R}^{d}$. With this notation we have $p(T,\cdot)\equiv p_{\mathbf{X}^{0}}(T,\cdot)$, and $p^{J}(\cdot)\equiv p_{\mathbf{X}^{0}}^{J}(\cdot)$.

\begin{theorem}[{\cite[Corollary 2.7]{bally}}]
\label{ballythm}
Let $J\in \mathbb{N}$, and $\{t_{j}\}_{j=0}^{J}\subset [0,T]$ be a mesh with uniform step size $\tau=\frac{T}{J}$. For all $\mathbf{y},\mathbf{x}\in \mathbb{R}^{d}$ we have
\begin{equation*}
\lvert p_{\mathbf{y}}(T,\mathbf{x})-p_{\mathbf{y}}^{J}(\mathbf{x})\rvert\leq C(T)\cdot \tau \cdot \exp\left(-c \frac{\lVert \mathbf{y}-\mathbf{x} \rVert_{\mathbb{R}^{d}}^{2}}{T} \right)\,,
\end{equation*}
where $c>0$, and $C(T)>0$ is a constant that depends on $T$.
\end{theorem}

\medskip

For the proof of Theorem \ref{ballythm} see \cite{bally}.

\medskip

\begin{theorem}
\label{modballythm}
Let $J\in \mathbb{N}$, and $\{t_{j}\}_{j=0}^{J}\subset [0,T]$ be a mesh with uniform step size $\tau=\frac{T}{J}$. Let $p(T,\cdot)$ and $p^{J}$ be the pdfs of $\mathbf{X}_{T}$ and $\mathbf{Y}^{J}$ from \eqref{eq:SDE} resp.~\eqref{eq:euler}. Then, for $1\leq q \leq \infty$,
\begin{equation*}
\lVert p(T,\cdot)-p^{J}\rVert_{\mathbb{L}^{q}}\leq C(T,d,q)\cdot \tau \,,
\end{equation*}
where $C(T,d,q)>0$ is a constant which depends on $T$, $d$, and $q$.
\end{theorem}
Consequently, for $1\leq q \leq \infty$,
\begin{equation*}
\lVert p(T,\cdot)-p^{J}\rVert_{\mathbb{L}^{q}} \to 0 \qquad \mathbb{P}-a.s. \qquad (\text{as }\tau \to 0)\,.
\end{equation*}

\begin{proof} 
By the {\em Riesz-Thorin interpolation theorem} for $\mathbb{L}^{q}$ spaces, it suffices to prove the statement for $q=1$ and $q=\infty$.\\
{\bf Step 1:} ($q=1$) By using the {\em Chapman-Kolmogorov equation} for the transition density, a calculation yields
\begin{align*}
\lVert p(T,\cdot)-p^{J}\rVert_{\mathbb{L}^{1}}&=\int\limits_{\mathbb{R}^{d}}\,\lvert p_{\mathbf{X}^{0}}(T,\mathbf{x})-p_{\mathbf{X}^{0}}^{J}(\mathbf{x})\rvert\,\mathrm{d}\mathbf{x}=\int\limits_{\mathbb{R}^{d}}\, \bigg \lvert \int\limits_{\mathbb{R}^{d}}\,p_{\mathbf{y}}(T,\mathbf{x})\cdot p_{0}(\mathbf{y})\,\mathrm{d}\mathbf{y} - \int\limits_{\mathbb{R}^{d}}\,p_{\mathbf{y}}^{J}(\mathbf{x})\cdot p_{0}(\mathbf{y})\,\mathrm{d}\mathbf{y}\,\bigg \rvert \, \mathrm{d}\mathbf{x} \notag \\
&= \int\limits_{\mathbb{R}^{d}}\, \bigg \lvert \int \limits_{\mathbb{R}^{d}} \, \left(p_{\mathbf{y}}(T,\mathbf{x})-p_{\mathbf{y}}^{J}(\mathbf{x})  \right)\cdot p_{0}(\mathbf{y})\,\mathrm{d}\mathbf{y}\,\bigg \rvert \, \mathrm{d}\mathbf{x} \leq \int\limits_{\mathbb{R}^{d}}\,  \int \limits_{\mathbb{R}^{d}} \, \big \lvert p_{\mathbf{y}}(T,\mathbf{x})-p_{\mathbf{y}}^{J}(\mathbf{x})  \big \rvert \cdot p_{0}(\mathbf{y})\,\mathrm{d}\mathbf{y} \, \mathrm{d}\mathbf{x}\,.
\end{align*}
Application of Theorem \ref{ballythm} then leads to
\begin{align*}
\lVert p(T,\cdot)-p^{J}\rVert_{\mathbb{L}^{1}}&\leq C(T)\cdot \tau \cdot \int\limits_{\mathbb{R}^{d}}\,\int\limits_{\mathbb{R}^{d}}\, \exp\left(-c \frac{\lVert \mathbf{y}-\mathbf{x} \rVert_{\mathbb{R}^{d}}^{2}}{T} \right)\cdot p_{0}(\mathbf{y})\,\mathrm{d}\mathbf{y}\,\mathrm{d}\mathbf{x} \notag \\
& = C(T)\cdot \tau \cdot \int\limits_{\mathbb{R}^{d}}\,p_{0}(\mathbf{y})\,\int\limits_{\mathbb{R}^{d}}\, \exp\left(-c \frac{\lVert \mathbf{y}-\mathbf{x} \rVert_{\mathbb{R}^{d}}^{2}}{T} \right)\,\mathrm{d}\mathbf{x}\,\mathrm{d}\mathbf{y}\,.
\end{align*}
Since
\begin{equation*}
\int\limits_{\mathbb{R}^{d}}\,\exp\left(-c\frac{\lVert \mathbf{x}-\mathbf{y}\rVert_{\mathbb{R}^{d}}^{2}}{T}\right) \,\mathrm{d}\mathbf{x}=\left( \frac{T \pi}{c} \right)^{\nicefrac{d}{2}}\,,
\end{equation*}
we further conclude
\begin{align*}
\lVert p(T,\cdot)-p^{J}\rVert_{\mathbb{L}^{1}}&\leq  C(T)\cdot \tau \cdot \int\limits_{\mathbb{R}^{d}}\,p_{0}(\mathbf{y})\cdot \left(\frac{T\pi}{c} \right)^{\nicefrac{d}{2}} \,\mathrm{d}\mathbf{y}  =C(T)\cdot \left(\frac{T\pi}{c} \right)^{\nicefrac{d}{2}}\cdot \tau\,,
\end{align*}
from which the convergence statement in the theorem follows for $q=1$.\\
{\bf Step 2:} ($q=\infty$) Similar to Step 1, we obtain
\begin{align*}
\lVert p(T,\cdot)-p^{J}\rVert_{\mathbb{L}^{\infty}}&=\sup\limits_{\mathbf{x} \in \mathbb{R}^{d}}\,\lvert p_{\mathbf{X}^{0}}(T,\mathbf{x})-p_{\mathbf{X}^{0}}^{J}(\mathbf{x})\rvert \leq \sup\limits_{\mathbf{x} \in \mathbb{R}^{d}}\,\int \limits_{\mathbb{R}^{d}} \, \big \lvert p_{\mathbf{y}}(T,\mathbf{x})-p_{\mathbf{y}}^{J}(\mathbf{x})  \big \rvert \cdot p_{0}(\mathbf{y})\,\mathrm{d}\mathbf{y}\\
&\leq C(T) \cdot \tau \cdot \sup\limits_{\mathbf{x} \in \mathbb{R}^{d}} \, \underbrace{\int \limits_{\mathbb{R}^{d}} \,\exp\left(-c \frac{\lVert \mathbf{y}-\mathbf{x} \rVert_{\mathbb{R}^{d}}^{2}}{T} \right)\cdot p_{0}(\mathbf{y})\,\mathrm{d}\mathbf{y}}_{\leq 1} \leq C(T) \cdot \tau \,,
\end{align*}
from which the convergence assertion follows for $q=\infty$.
\end{proof}

\

\section{Data-dependent density estimators}
\label{3}


In our opinion, {\em data-dependent} partitioning strategies in combination with histogram-based density estimators, which are structured as decision trees, are an intriguing prospect to address the approximation of \eqref{eq:fpe}; see Figure \ref{tree} and Subsection \ref{3.1}. In Subsection \ref{3.2}, we verify the $\mathbb{P}-a.s.$ $\mathbb{L}^{1}-$convergence of the estimator $\widehat{\mathtt{p}}^{J,\mathtt{M}}\equiv \widehat{\mathtt{p}}^{J,\mathtt{M}}_{\mathtt{Ges}}$ in Algorithm \ref{strategy}; see Theorem \ref{thmgessamanconvergence}.

\subsection{On data-dependent partitioning strategies}
\label{3.1}

We present two splitting rules for step 3) in Algorithm \ref{strategy} to obtain two {\em data-dependent} estimators $\widehat{\mathtt{p}}^{J,\mathtt{M}}_{\mathtt{Ges}}$ and $\widehat{\mathtt{p}}^{J,\mathtt{M}}_{\mathtt{BTC}}$. For their construction we grow ($\mathtt{N}-$ary resp.~binary) trees of different heights, where interior nodes represent an ongoing refinement of the partitioning of $\mathbb{R}^{d}$, which is terminated when the maximum height is reached to explain the structure of the complex data $\mathbf{D}_{\mathtt{M}}^{J}$; see Figure \ref{partitions}.


\begin{figure}[h!]
\centering
\begin{minipage}[t]{.25\textwidth}
\centering
\includegraphics[scale=0.38]{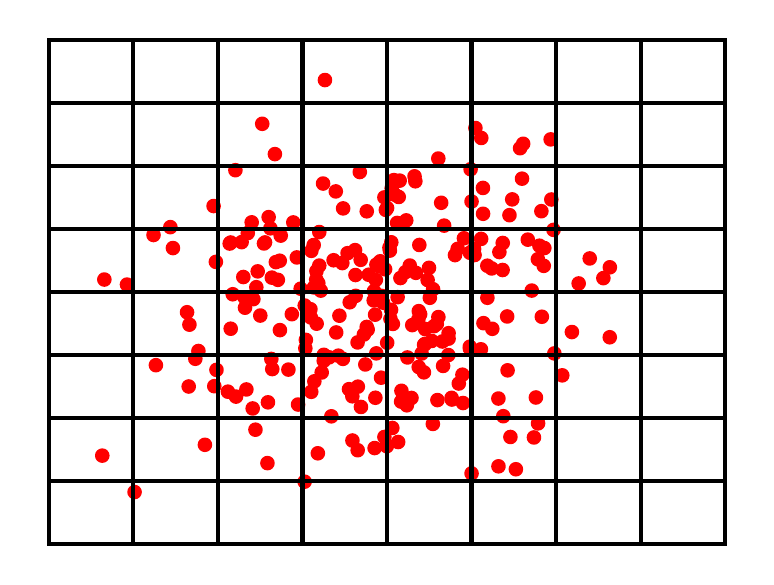}
\captionsetup{labelfont={bf}}
\subcaption{Uniform partition $\pmb{\mathcal{P}}_{\mathtt{uni}}^{J,\mathtt{M}}$}
\end{minipage}
\hspace{.08\linewidth}
\begin{minipage}[t]{.25\textwidth}
\centering
\includegraphics[scale=0.38]{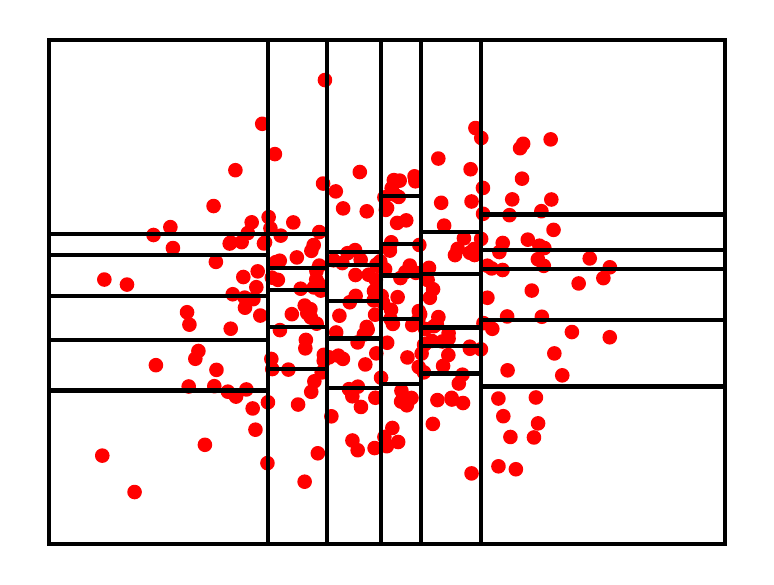}
\captionsetup{labelfont={bf}}
\subcaption{{\em Data-dependent} partition $\pmb{\mathcal{P}}_{\mathtt{Ges}}^{J,\mathtt{M}}$}
\end{minipage}
\hspace{.08\linewidth}
\begin{minipage}[t]{.25\textwidth}
\centering
\includegraphics[scale=0.38]{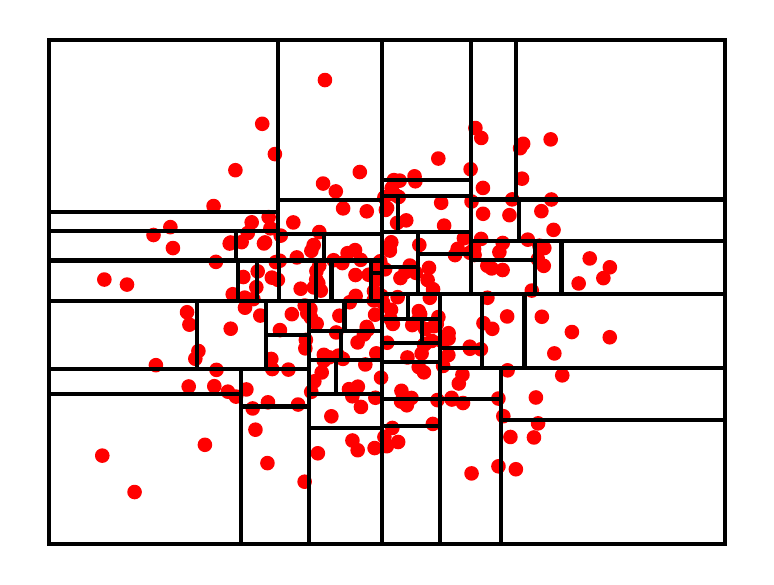}
\captionsetup{labelfont={bf}}
\subcaption{{\em Data-dependent} partition $\pmb{\mathcal{P}}_{\mathtt{BTC}}^{J,\mathtt{M}}$}
\end{minipage}

\caption{Partitioning of $\mathbb{R}^{d}$, where the sample is drawn from a $2-$dimensional standard normal distribution.}
\label{partitions}

\end{figure}

\subsubsection{Gessaman's splitting rule}
\label{3.1.1}

We fix $\mathtt{M}\in \mathbb{N}$, and let $\mathbf{D}_{\mathtt{M}}^{J}:=\{\mathbf{Y}^{J,m}\}_{m=1}^{\mathtt{M}}$ be an {\em i.i.d.}~$\mathbb{R}^{d}-$valued sample set from \eqref{eq:euler} with corresponding density $p^{J}$ on $\mathbb{R}^{d}$. Let $k_{\mathtt{M}}\in \mathbb{N}$ denote the number of realizations out of $\mathbf{D}_{\mathtt{M}}^{J}$ of each fully refined cell $\mathbf{R}_{r}\in \pmb{\mathcal{P}}^{J,\mathtt{M}}_{\mathtt{Ges}}$.

\medskip

For Gessaman's rule, $\mathbb{R}^{d}$ is partitioned into $\mathtt{N}$ sets using hyperplanes perpendicular to the $x_{1}-$axis, such that each set contains an equal number of samples. This produces $\mathtt{N}$ sets resembling long strips. In the same way, one partitions each of these $\mathtt{N}$ long strips along the $x_{2}-$axis into further $\mathtt{N}$ cells and continues in such a way in all coordinates until $\mathtt{N}^{d}$ cells are reached in total, each containing the same number of sample iterates. More precisely, Gessaman's rule describes the following algorithm.

\begin{gessaman rule}[see \cite{gessaman,lugosi}]
\label{gessamanstrategy}
Select $\mathtt{M}$ and $k_{\mathtt{M}}$ such that $\mathtt{N}=\Big\lceil \left(\frac{\mathtt{M}}{k_{\mathtt{M}}}\right)^{\nicefrac{1}{d}} \Big \rceil\in \mathbb{N}$. Let $\mathcal{R}^{\mathtt{M}}_{0,0}:=\mathbb{R}^{d}$.\\
{\bf For} $i=0,...,d$ do:\\
~\text{~~~}~{\bf For} $j=0,...,\mathtt{N}^{i}-1$ do:
\begin{itemize}
\item[~] Decompose $\mathcal{R}^{\mathtt{M}}_{i,j}$ into $\mathtt{N}$ hyper-rectangles $\mathcal{R}^{\mathtt{M}}_{i+1,dj+\ell}=\mathcal{R}^{\mathtt{M}}_{i,j}\cap \left(\mathbb{R}^{i} \times [x_{\ell},x_{\ell +1})\times \mathbb{R}^{d-1-i}   \right)$ with $x_{\ell}\in \mathbb{R}$ such that $-\infty=x_{0}<x_{1}<...<x_{\mathtt{N}-1}<x_{\mathtt{N}}=\infty$ and each $\mathcal{R}^{\mathtt{M}}_{i+1,dj+\ell}$ contains $\frac{\mathtt{M}}{\mathtt{N}^{i+1}}$ samples. Set $\pmb{\mathcal{P}}_{\mathtt{Ges}}^{J,\mathtt{M}}= \{\mathcal{R}^{\mathtt{M}}_{d,j}\,:\, 0\leq j \leq \mathtt{R}_{\mathtt{Ges}}-1\}$, where $\mathtt{R}_{\mathtt{Ges}}=\mathtt{N}^{d}$.
\end{itemize}

\end{gessaman rule}

The following example clarifies the construction.

\begin{example}
\label{gessamanclarify}
Consider Figure \ref{partitions} (b) in combination with Figure \ref{tree} (a). Starting with $\mathcal{R}^{\mathtt{M}}_{0,0}:=\mathbb{R}^{d}$, we observe $\mathtt{N}=6$ long strips in $x_{1}-$direction according to Gessaman's rule \ref{gessamanstrategy}. For $j=0,..,\mathtt{N}-1$, we denote by $\mathcal{R}^{\mathtt{M}}_{1,j}$ the $j-$th long strip. Then, each of these long strips is further decomposed into $\mathtt{N}=6$ cells. The total amount of cells is $\mathtt{R}_{\mathtt{Ges}}=\mathtt{N}^{2}=36$. For $r=0,...,\mathtt{N}^{2}-1$, we denote by $\mathcal{R}^{\mathtt{M}}_{2,r}\equiv \mathbf{R}_{r}$ the $r-$th cell  in the generated partition $\pmb{\mathcal{P}}_{\mathtt{Ges}}^{J,\mathtt{M}}=\{\mathbf{R}_{r};\,r=0,...,\mathtt{R}_{\mathtt{Ges}}-1\}$ with $\mathtt{R}_{\mathtt{Ges}}=\mathtt{N}^{2}$.
\end{example}

\medskip


\subsubsection{A Binary Tree Cuboid (BTC)-based splitting rule}
\label{3.1.2}

In this subsection, we present an alternative rule leading to trees with a considerably larger depth of $\mathcal{O}(\log(\mathtt{M}))$ and better resolution of the density $p^{J}$. It is a modification of the BTC splitting rule in \cite{dunst}: In view of condition (c) in Theorem \ref{thmlugosi}, longest sides of cells are split recursively to ensure the shrinking of cell diameters. To this end, let $\mathtt{M}=2^{N}$ for some $N\in \mathbb{N}$, and fix a number $k_{\mathtt{M}}=2^{n}$, where $n<N$. The generated partition is denoted by $\pmb{\mathcal{P}}_{\mathtt{BTC}}^{J,\mathtt{M}}$.

\begin{BTC-rule}
\label{moddunststrategy}
Set $\kappa=N-n$, and $\mathcal{S}_{0,0}=\mathbf{D}_{\mathtt{M}}^{J}$, and $\mathcal{R}_{0,0}^{\mathtt{M}}=\mathbb{R}^d$.\\ 
{\bf For} $p=0,...,\kappa-1$ do:\\
~\text{~~~}~{\bf For} $q=0,...,2^{p}-1$ do:
\begin{itemize}
\item[(I)] Set $\pmb{\mathcal{S}}:=\mathcal{S}_{p,q}$ and $\pmb{\mathcal{R}}:=\mathcal{R}_{p,q}^\mathtt{M}$.
\item[(II)] Find the index of the coordinate axes $\ell \in \{1,...,d\}$ parallel to a largest edge of the cell $\pmb{\mathcal{R}}$; denote this component by $\ell_{p,q}\in \{1,...,d\}$. If two or more edges have equal largest length, choose the edge with the lowest index $\ell$. Edge lengths in a direction in which the cell is unbounded are considered infinite.
\item[(III)] Compute the median $\mathtt{med}_{p,q}\in \mathbb{R}$ along the $\ell_{p,q}-$axis of all sample iterates in $\pmb{\mathcal{S}}$.
\item[(IV)] 
Divide $\pmb{\mathcal{R}}=\mathcal{R}_{p+1,2q}^\mathtt{M} \cup \mathcal{R}_{p+1,2q+1}^\mathtt{M}$ into two disjoint cells by a hyperplane perpendicular to the $\ell_{p,q}-$axis at $\mathtt{med}_{p,q}$, and set $\mathcal{S}_{p+1,2q}=\pmb{\mathcal{S}}\cap \mathcal{R}_{p+1,2q}^\mathtt{M}$, $\mathcal{S}_{p+1,2q+1}=\pmb{\mathcal{S}}\cap \mathcal{R}_{p+1,2q+1}^\mathtt{M}$.
\end{itemize}
\end{BTC-rule}

Note that \q{$\mathcal{S}_{p,q}$} contains the samples in \q{$\pmb{\mathcal{S}}$} which are contained in \q{$\mathcal{R}_{p,q}^{\mathtt{M}}$}. The partition $\pmb{\mathcal{P}}_{\mathtt{BTC}}^{J,\mathtt{M}}$ generated via this procedure has a total number of $\mathtt{R}_{\mathtt{BTC}}:=2^{\kappa}=2^{N-n}=\frac{\mathtt{M}}{k_{\mathtt{M}}}$ cells.

\subsection{Strong consistency of the density estimator $\widehat{\mathtt{p}}_{\mathtt{Ges}}^{J,\mathtt{M}}$}
\label{3.2}

The goal of this subsection is to prove the strong $\mathbb{L}^{1}-$consistency of the estimator $\widehat{\mathtt{p}}^{J,\mathtt{M}}_{\mathtt{Ges}}$ from Algorithm \ref{strategy} with Gessaman's rule \ref{gessamanstrategy}; see Theorem \ref{thmgessamanconvergence} below. The combination with Theorem \ref{modballythm} then settles convergence of the estimator towards $p(T,\cdot)$ through \eqref{eq:convergence}.

The density estimator $\widehat{\mathtt{p}}^{J,\mathtt{M}}_{\mathtt{Ges}}$ was first proposed in \cite{gessaman}, where its strong $\mathbb{L}^{1}-$consistency was motivated. This section aims to provide a proof of this property. It rests on the verification of three criteria given in \cite[Theorem 1]{lugosi} to guarantee strong $\mathbb{L}^{1}-$consistency for a general density estimator; see also Theorem \ref{thmlugosi} below.

Let $\mathtt{M}\in \mathbb{N}$. We write $\pmb{\mathcal{P}}^{J,\mathtt{M}}(\mathbf{x}_{1},...,\mathbf{x}_{\mathtt{M}})$ for a partition of $\mathbb{R}^{d}$ which is realised based upon given points $\mathbf{x}_{1},...,\mathbf{x}_{\mathtt{M}}\in \mathbb{R}^{d}$. Note that with this notation we have $\pmb{\mathcal{P}}^{J,\mathtt{M}}\equiv \pmb{\mathcal{P}}^{J,\mathtt{M}}(\mathbf{Y}^{J,1},...,\mathbf{Y}^{J,\mathtt{M}})$ based on the sample iterates $\mathbf{D}^{J}_{\mathtt{M}}$ from Algorithm \ref{strategy}. We denote by $\pmb{\mathcal{A}}_{\mathtt{M}}:=\{\pmb{\mathcal{P}}^{J,\mathtt{M}}(\mathbf{x}_{1},...,\mathbf{x}_{\mathtt{M}})\,:\,\mathbf{x}_{1},...,\mathbf{x}_{\mathtt{M}}\in \mathbb{R}^{d}\}$ the collection of nonrandom partitions following the same underlying partitioning strategy. Moreover, we denote by $\mathtt{m}(\pmb{\mathcal{A}}_{\mathtt{M}}):=\sup\limits_{\pmb{\mathcal{P}}\in \pmb{\mathcal{A}}_{\mathtt{M}}} \#\{\text{cells in }\pmb{\mathcal{P}}\}$ the maximal cell count of $\pmb{\mathcal{A}}_{\mathtt{M}}$.\\
Fix $\mathtt{M}$ points $\tilde{\mathbf{x}}_{1},...,\tilde{\mathbf{x}}_{\mathtt{M}}\in\mathbb{R}^{d}$ and let $\pmb{\mathcal{B}}:=\{\tilde{\mathbf{x}}_{1},...,\tilde{\mathbf{x}}_{\mathtt{M}}\}$. We denote by $\Delta \bigl(\pmb{\mathcal{A}}_{\mathtt{M}},(\tilde{\mathbf{x}}_{1},...,\tilde{\mathbf{x}}_{\mathtt{M}})\bigr)$ the number of distinct partitions 
\begin{equation*}
\{\mathbf{R}_{0}\cap \pmb{\mathcal{B}},...,\mathbf{R}_{\mathtt{R}-1}\cap \pmb{\mathcal{B}}\}
\end{equation*}
of the finite set $\pmb{\mathcal{B}}$ that are induced by partitions $\pmb{\mathcal{P}}:=\{\mathbf{R}_{0},...,\mathbf{R}_{\mathtt{R}-1}\}\in \pmb{\mathcal{A}}_{\mathtt{M}}$. Then, we define the growth function of $\pmb{\mathcal{A}}_{\mathtt{M}}$ as
\begin{equation*}
\Delta_{\mathtt{M}}^{\ast}(\pmb{\mathcal{A}}_{\mathtt{M}}):=\max\limits_{(\mathbf{x}_{1},...,\mathbf{x}_{\mathtt{M}})\in \mathbb{R}^{\mathtt{M}\cdot d}} \Delta \bigl(\pmb{\mathcal{A}}_{\mathtt{M}},(\mathbf{x}_{1},...,\mathbf{x}_{\mathtt{M}})\bigr)\,,
\end{equation*}
which is the largest number of distinct partitions of any $\mathtt{M}$ point subset of $\mathbb{R}^{d}$ that can be induced by the partitions in $\pmb{\mathcal{A}}_{\mathtt{M}}$.\\
Finally, for every $\mathbf{x}\in \mathbb{R}^{d}$, we denote by $\pmb{\mathcal{P}}^{J,\mathtt{M}}[\mathbf{x}]$ the unique cell in the partition $\pmb{\mathcal{P}}^{J,\mathtt{M}}$ that contains the point $\mathbf{x}$.

\medskip

\begin{theorem}[{\cite[Theorem 1]{lugosi}}]
\label{thmlugosi}
Let $\mathbf{Y}^{J,1},\mathbf{Y}^{J,2},...$ be {\em i.i.d.}~random vectors in $\mathbb{R}^{d}$ whose common distribution $\mathbb{P}_{\mathbf{Y}^{J}}$ has density $p^{J}$. Let $\mathtt{M}\in \mathbb{N}$. Let $\pmb{\mathcal{A}}_{\mathtt{M}}:=\{\pmb{\mathcal{P}}^{J,\mathtt{M}}(\mathbf{x}_{1},...,\mathbf{x}_{\mathtt{M}})\,:\,\mathbf{x}_{1},...,\mathbf{x}_{\mathtt{M}}\in \mathbb{R}^{d}\}$ be the collection of partitions following the same underlying partitioning strategy. If, as $\mathtt{M}$ tends to infinity
\begin{itemize}
\item[(a)] $\mathtt{M}^{-1}\cdot \mathtt{m}(\pmb{\mathcal{A}}_{\mathtt{M}})\to 0$
\item[(b)] $\mathtt{M}^{-1}\cdot \ln \bigl( \Delta_{\mathtt{M}}^{\ast}(\pmb{\mathcal{A}}_{\mathtt{M}})\bigr) \to 0$
\item[(c)] $\mathbb{P}_{\mathbf{Y}^{J}}\big[ \{\mathbf{x}\in \mathbb{R}^{d}\,:\, \mathtt{diam}\bigl(\pmb{\mathcal{P}}^{J,\mathtt{M}}[\mathbf{x}]\bigr)>\gamma \} \big] \to 0 \quad \mathbb{P}-a.s.$ for every $\gamma>0$\,,
\end{itemize}
then the estimator $\widehat{\mathtt{p}}^{J,\mathtt{M}}$ given in \eqref{eq:estimator} is strongly $\mathbb{L}^{1}-$consistent, i.e.,
\begin{equation*}
\lVert p^{J} -\widehat{\mathtt{p}}^{J,\mathtt{M}} \rVert_{\mathbb{L}^{1}} \to 0 \qquad \mathbb{P}-a.s. \qquad \text{(as } \mathtt{M}\to \infty)\,.
\end{equation*}
\end{theorem}

\medskip

The next theorem states the strong $\mathbb{L}^{1}-$consistency of the estimator $\widehat{\mathtt{p}}^{J,\mathtt{M}}_{\mathtt{Ges}}$ of Algorithm \ref{strategy}, where step 3) is realised via  Scheme \ref{gessamanstrategy}.

\begin{theorem}
\label{thmgessamanconvergence}
Let $J\in \mathbb{N}$, and $\{t_{j}\}_{j=0}^{J}\subset [0,T]$ be a mesh with uniform step size $\tau=\frac{T}{J}$. Let $\mathtt{M}\in \mathbb{N}$ and $k_{\mathtt{M}}\in \mathbb{N}$ with $k_{\mathtt{M}}<\mathtt{M}$. Let $\mathbf{D}_{\mathtt{M}}^{J}:=\{\mathbf{Y}^{J,m}\}_{m=1}^{\mathtt{M}}$ be the sample set from \eqref{eq:euler}. If $k_{\mathtt{M}}\to \infty$ and $\frac{k_{\mathtt{M}}}{\mathtt{M}}\to 0$ as $\mathtt{M}\to \infty$, then {\em Gessaman's} estimator $\widehat{\mathtt{p}}^{J,\mathtt{M}}_{\mathtt{Ges}}$ in Algorithm \ref{strategy}, where step 3) is realised via  Scheme \ref{gessamanstrategy}, is strongly $\mathbb{L}^{1}-$consistent, i.e.,
\begin{equation*}
\lVert p^{J}-\widehat{\mathtt{p}}^{J,\mathtt{M}}_{\mathtt{Ges}}\rVert_{\mathbb{L}^{1}}\to 0 \qquad \mathbb{P}-a.s. \qquad \text{(as } \mathtt{M}\to \infty)\,.
\end{equation*}
\end{theorem}

\medskip

We prove Theorem \ref{thmgessamanconvergence} by verifying the conditions (a), (b), and (c) in Theorem \ref{thmlugosi}. As we will see, the conditions (a) and (b) automatically follow from the concept of {\em statistically equivalent cells} in combination with an appropriate choice of $\mathtt{M}$ and $k_{\mathtt{M}}$. Here, we closely follow the arguments of \cite[Theorem 4]{lugosi}, which verifies the conditions (a), (b), and (c) of Theorem \ref{thmlugosi} in the case of $d=1$, and therefore proves Theorem \ref{thmgessamanconvergence} for $d=1$.
The verification of condition (c) for $d\geq 2$, however, is more delicate, and we are not aware of any results in the literature where this part is rigorously proved for Gessaman's rule \ref{gessamanstrategy}.

\medskip

\begin{proof}[Proof of Theorem \ref{thmgessamanconvergence}.]
Let $d\geq 2$.\\
{\bf Verification of (a) in Theorem \ref{thmlugosi}:} Since we are in the setting of statistically equivalent cells, the number of cells in every partition in $\pmb{\mathcal{A}}_{\mathtt{M}}:=\{\pmb{\mathcal{P}}^{J,\mathtt{M}}_{\mathtt{Ges}}(\mathbf{z}_{1},...,\mathbf{z}_{\mathtt{M}})\,:\,\mathbf{z}_{1},...,\mathbf{z}_{\mathtt{M}}\in \mathbb{R}^{d}\}$ is equal to $\mathtt{R}_{\mathtt{Ges}}:=\mathtt{N}^{d}:=\Big\lceil \left(\frac{\mathtt{M}}{k_{\mathtt{M}}}\right)^{\nicefrac{1}{d}} \Big \rceil^{d}$; see Gessaman's rule \ref{gessamanstrategy}. Hence, due to the growth of $k_{\mathtt{M}}$ related to $\mathtt{M}$, we get
\begin{equation*}
\mathtt{M}^{-1}\cdot \mathtt{m}(\pmb{\mathcal{A}}_{\mathtt{M}})=\frac{1}{\mathtt{M}} \cdot \Big\lceil \left(\frac{\mathtt{M}}{k_{\mathtt{M}}}\right)^{\nicefrac{1}{d}} \Big \rceil^{d} \to 0 \qquad (\text{as } \mathtt{M}\to \infty)\,.
\end{equation*}
{\bf Verification of (b) in Theorem \ref{thmlugosi}:} A combinatorial argument immediately yields ({\em cf.}~\cite[Subsec.~6.3]{lugosi})
\begin{equation}
\label{combiargument}
\Delta_{\mathtt{M}}^{\ast}(\pmb{\mathcal{A}}_{\mathtt{M}})=\binom{\mathtt{M}+\mathtt{R}_{\mathtt{Ges}}}{\mathtt{M}}^{d}\,.
\end{equation}
We consider the {\em binary entropy function} $h(x)=-x\ln(x)-(1-x)\ln(1-x)$, $x\in (0,1)$. We see that $h$ is increasing on $(0,\nicefrac{1}{2}]$, symmetric about $\nicefrac{1}{2}$, and $h(x)\to 0$ as $x\to 0$. Using \eqref{combiargument}, the inequality
\begin{equation*}
\ln \left( \binom{s}{t} \right)\leq s h\left( \nicefrac{t}{s}\right)\,,\qquad s,t \in \mathbb{N}
\end{equation*}
stated in \cite{koerner}, and the fact that $h$ is symmetric about $\nicefrac{1}{2}$, we get
\begin{align*}
\ln \bigl( \Delta_{\mathtt{M}}^{\ast}(\pmb{\mathcal{A}}_{\mathtt{M}})\bigr)&=\ln \left( \binom{\mathtt{M}+\mathtt{R}_{\mathtt{Ges}}}{\mathtt{M}}^{d} \right)=d \cdot \ln \left( \binom{\mathtt{M}+\mathtt{R}_{\mathtt{Ges}}}{\mathtt{M}} \right) \leq d\cdot (\mathtt{M}+\mathtt{R}_{\mathtt{Ges}})\cdot h\left( \frac{\mathtt{M}}{\mathtt{M}+\mathtt{R}_{\mathtt{Ges}}}\right)\\
& = d\cdot(\mathtt{M}+\mathtt{R}_{\mathtt{Ges}})\cdot h\left( \frac{\mathtt{R}_{\mathtt{Ges}}}{\mathtt{M}+\mathtt{R}_{\mathtt{Ges}}}\right)\,.
\end{align*}
Since $\mathtt{M}\geq \mathtt{R}_{\mathtt{Ges}}$, $\frac{\mathtt{R}_{\mathtt{Ges}}}{\mathtt{M}+\mathtt{R}_{\mathtt{Ges}}}\leq \frac{1}{k_{\mathtt{M}}}$, we obtain
\begin{equation*}
\ln \bigl( \Delta_{\mathtt{M}}^{\ast}(\pmb{\mathcal{A}}_{\mathtt{M}})\bigr)\leq 2d\cdot \mathtt{M}\cdot h\left( \frac{1}{k_{\mathtt{M}}}\right)\,.
\end{equation*}
Hence, using the properties of $h$ and that $k_{\mathtt{M}}\to \infty$ as $\mathtt{M}\to \infty$, we conclude
\begin{equation*}
\mathtt{M}^{-1}\cdot \ln \bigl( \Delta_{\mathtt{M}}^{\ast}(\pmb{\mathcal{A}}_{\mathtt{M}})\bigr)\leq 2d \cdot h\left( \frac{1}{k_{\mathtt{M}}}\right) \to 0 \qquad (\text{as } \mathtt{M}\to \infty)\,.
\end{equation*}
{\bf Verification of (c) in Theorem \ref{thmlugosi}:} Fix $\gamma>0$, $\varepsilon>0$, and let $L\equiv L(\varepsilon)>0$ be such that $\mathbb{P}_{\mathbf{Y}^{J}}[ \mathbb{R}^{d}\setminus [-L,L]^{d}]\leq \varepsilon$. 
Because all cells in the partition $\pmb{\mathcal{P}}^{J,\mathtt{M}}_{\mathtt{Ges}}$ are disjoint excluding the boundary, and due to the choice of $L>0$ above, we find
\begin{align}
\label{c1}
\mathbb{P}_{\mathbf{Y}^{J}}\big[ &\{\mathbf{x}\in \mathbb{R}^{d}\,:\, \mathtt{diam}\bigl(\pmb{\mathcal{P}}^{J,\mathtt{M}}_{\mathtt{Ges}}[\mathbf{x}]\bigr)>\gamma \} \big]\leq \mathbb{P}_{\mathbf{Y}^{J}}\big[ \{\mathbf{x}\in \mathbb{R}^{d}\,:\, \mathtt{diam}\bigl(\pmb{\mathcal{P}}^{J,\mathtt{M}}_{\mathtt{Ges}}[\mathbf{x}]\bigr)>\gamma \} \cap [-L,L]^{d}\big] +\varepsilon \notag\\
&=\mathbb{P}_{\mathbf{Y}^{J}}\Bigg[   \bigcup\limits_{\substack{r\in \{0,...,\mathtt{R}_{\mathtt{Ges}}-1\}: \\  \mathtt{diam}(\mathbf{R}_{r})\geq\gamma}} \mathbf{R}_{r} \cap [-L,L]^{d} \Bigg] + \varepsilon \notag \\
& =  \sum\limits_{r=0}^{\mathtt{R}_{\mathtt{Ges}}-1} \mathbbm{1}_{\{\mathtt{diam}(\mathbf{R}_{r})\geq \gamma \}}\cdot \mathbb{P}_{\mathbf{Y}^{J}}\big[ \mathbf{R}_{r}\cap [-L,L]^{d}\big] +\varepsilon \notag \\
&\leq  \sum\limits_{\substack{r\in \{0,...,\mathtt{R}_{\mathtt{Ges}}-1\}: \\  \mathbf{R}_{r}\cap [-L,L]^{d}\neq \emptyset}} \mathbbm{1}_{\{\mathtt{diam}(\mathbf{R}_{r})\geq \gamma \}} \cdot \left(  \mathbb{P}_{\mathbf{Y}^{J}}[\mathbf{R}_{r}]-\frac{k_{\mathtt{M}}}{\mathtt{M}} + \frac{k_{\mathtt{M}}}{\mathtt{M}} \right)  +\varepsilon \notag \\
&\leq \frac{k_{\mathtt{M}}}{\mathtt{M}}\cdot \sum\limits_{\substack{r\in \{0,...,\mathtt{R}_{\mathtt{Ges}}-1\}:  \\ \mathbf{R}_{r}\cap [-L,L]^{d}\neq \emptyset}} \mathbbm{1}_{\{\mathtt{diam}(\mathbf{R}_{r})\geq \gamma \}} + \sum\limits_{\substack{r\in \{0,...,\mathtt{R}_{\mathtt{Ges}}-1\}:  \\ \mathbf{R}_{r}\cap [-L,L]^{d}\neq \emptyset}} \mathbbm{1}_{\{\mathtt{diam}(\mathbf{R}_{r})\geq \gamma \}} \cdot \Big \lvert \mathbb{P}_{\mathbf{Y}^{J}}[ \mathbf{R}_{r}]-\frac{k_{\mathtt{M}}}{\mathtt{M}}  \Big \rvert + \varepsilon\,.
\end{align}
Since each cell \q{$\mathbf{R}_{r}$} is a hyper-rectangle, we can write $\mathbf{R}_{r}=\prod\limits_{i=1}^{d} [a_{i}^{(r)},b_{i}^{(r)}]$ for $a_{i}^{(r)}<b_{i}^{(r)}$. If $\mathtt{diam}(\mathbf{R}_{r}):=\sqrt{\sum\limits_{i=1}^{d} \lvert b_{i}^{(r)}-a_{i}^{(r)}\rvert^{2}}$ is greater or equal to $\gamma$, then its longest side is greater or equal to $\frac{\gamma}{\sqrt{d}}$, i.e., $\max\limits_{i=1,...,d} b_{i}^{(r)}-a_{i}^{(r)} \geq \frac{\gamma}{\sqrt{d}}$.
Because the second sum on the right-hand side of \eqref{c1} is bounded from above by 
\begin{equation*}
\mathbf{E}_{\mathtt{M}}:=\sum\limits_{r=0}^{\mathtt{R}_{\mathtt{Ges}}-1}  \Big \lvert \mathbb{P}_{\mathbf{Y}^{J}}[ \mathbf{R}_{r}]-\frac{k_{\mathtt{M}}}{\mathtt{M}}  \Big \rvert\,,
\end{equation*}
we therefore conclude from \eqref{c1}
\begin{align}
\label{c2}
\mathbb{P}_{\mathbf{Y}^{J}}\big[ &\{\mathbf{x}\in \mathbb{R}^{d}\,:\, \mathtt{diam}\bigl(\pmb{\mathcal{P}}^{J,\mathtt{M}}_{\mathtt{Ges}}[\mathbf{x}]\bigr)>\gamma \} \big]&\leq \frac{k_{\mathtt{M}}}{\mathtt{M}}\cdot \sum\limits_{\substack{r\in \{0,...,\mathtt{R}_{\mathtt{Ges}}-1\}: \\ \mathbf{R}_{r}\cap [-L,L]^{d}\neq \emptyset}} \mathbbm{1}_{\Big\{\max\limits_{i=1,...,d} b_{i}^{(r)}-a_{i}^{(r)} \geq \frac{\gamma}{\sqrt{d}} \Big\}} + \mathbf{E}_{\mathtt{M}} + \varepsilon\,.
\end{align}
In the following, we estimate the first sum on the right-hand side of \eqref{c2}. Standard calculations lead to
\begin{align}
\label{c3}
\sum\limits_{\substack{r\in \{0,...,\mathtt{R}_{\mathtt{Ges}}-1\}:  \\ \mathbf{R}_{r}\cap [-L,L]^{d}\neq \emptyset}}& \mathbbm{1}_{\big\{\max\limits_{i=1,...,d} b_{i}^{(r)}-a_{i}^{(r)}\geq \tfrac{\gamma}{\sqrt{d}} \big\}} \notag \\
&\leq \sum\limits_{\substack{r\in \{0,...,\mathtt{R}_{\mathtt{Ges}}-1\}:  \\ \mathbf{R}_{r}\subseteq [-L,L]^{d}}} \mathbbm{1}_{\big\{\max\limits_{i=1,...,d} b_{i}^{(r)}-a_{i}^{(r)}\geq \tfrac{\gamma}{\sqrt{d}} \big\}} + \sum\limits_{\substack{r\in \{0,...,\mathtt{R}_{\mathtt{Ges}}-1\}:  \\ \mathbf{R}_{r}\cap \left(\mathbb{R}^{d} \setminus [-L,L]^{d}\right)\neq \emptyset}} \mathbbm{1}_{\big\{\max\limits_{i=1,...,d} b_{i}^{(r)}-a_{i}^{(r)}\geq \tfrac{\gamma}{\sqrt{d}} \big\}} \notag \\
&\leq \#\Big\{r=0,...,\mathtt{R}_{\mathtt{Ges}}-1\,:\,\mathbf{R}_{r}\subseteq [-L,L]^{d},\,  \max\limits_{i=1,...,d} b_{i}^{(r)}-a_{i}^{(r)}\geq \tfrac{\gamma}{\sqrt{d}}  \Big\} \notag \\
& \qquad \qquad + \#\Big\{r=0,...,\mathtt{R}_{\mathtt{Ges}}-1\,:\,\mathbf{R}_{r}\cap \left(\mathbb{R}^{d} \setminus [-L,L]^{d}\right)\neq \emptyset  \Big\} \notag \\
&\leq \sum\limits_{i=1}^{d}\#\Big\{r=0,...,\mathtt{R}_{\mathtt{Ges}}-1\,:\,\mathbf{R}_{r}\subseteq [-L,L]^{d},\,  b_{i}^{(r)}-a_{i}^{(r)}\geq \tfrac{\gamma}{\sqrt{d}}  \Big\} \notag \\
& \qquad \qquad + \#\Big\{r=0,...,\mathtt{R}_{\mathtt{Ges}}-1\,:\,\mathbf{R}_{r}\cap \left(\mathbb{R}^{d} \setminus [-L,L]^{d}\right)\neq \emptyset  \Big\}\,.
\end{align}
Since there are at most $\frac{2L}{\nicefrac{\gamma}{\sqrt{d}}}$ disjoint intervals of length greater than $\nicefrac{\gamma}{\sqrt{d}}$ in $[-L,L]$, and the number of cuts along any coordinate axis in Gessaman's rule \ref{gessamanstrategy} is at most $\mathtt{N}:=\Big\lceil \left(\frac{\mathtt{M}}{k_{\mathtt{M}}}\right)^{\nicefrac{1}{d}} \Big \rceil$, we can bound the first sum in \eqref{c3} by $\sum\limits_{i=1}^{d} \frac{2L}{\nicefrac{\gamma}{\sqrt{d}}} \cdot \mathtt{N}= d \cdot \sqrt{d} \cdot \frac{2L}{\gamma} \cdot \Big\lceil \left(\frac{\mathtt{M}}{k_{\mathtt{M}}}\right)^{\nicefrac{1}{d}} \Big \rceil $ from above. Furthermore, since $[-L, L]^{d}$ has $2d$ sides, we can bound the second summand in \eqref{c3} by $2d\cdot \Big\lceil \left(\frac{\mathtt{M}}{k_{\mathtt{M}}}\right)^{\nicefrac{1}{d}} \Big \rceil$ from above. Thus, we conclude
\begin{equation}
\label{c4}
\sum\limits_{\substack{r\in \{0,...,\mathtt{R}_{\mathtt{Ges}}-1\}:  \\ \mathbf{R}_{r}\cap [-L,L]^{d}\neq \emptyset}} \mathbbm{1}_{\big\{\max\limits_{i=1,...,d} b_{i}^{(r)}-a_{i}^{(r)}\geq \tfrac{\gamma}{\sqrt{d}} \big\}}\leq \left(\sqrt{d}\cdot d \cdot \frac{2L}{\gamma} + 2d\right)\cdot \Big\lceil \left(\frac{\mathtt{M}}{k_{\mathtt{M}}}\right)^{\nicefrac{1}{d}} \Big \rceil\,.
\end{equation}
Hence, combining \eqref{c4} with \eqref{c2}, we get
\begin{equation}
\label{c5}
\mathbb{P}_{\mathbf{Y}^{J}}\big[ \{\mathbf{x}\in \mathbb{R}^{d}\,:\, \mathtt{diam}\bigl(\pmb{\mathcal{P}}^{J,\mathtt{M}}_{\mathtt{Ges}}[\mathbf{x}]\bigr)>\gamma \} \big]\leq  \frac{k_{\mathtt{M}}}{\mathtt{M}}\cdot \left(\sqrt{d}\cdot d \cdot \frac{2L}{\gamma} + 2d\right)\cdot \Big\lceil \left(\frac{\mathtt{M}}{k_{\mathtt{M}}}\right)^{\nicefrac{1}{d}} \Big \rceil + \mathbf{E}_{\mathtt{M}} +\varepsilon\,.
\end{equation}
Since (a) and (b) in Theorem \ref{thmlugosi} are verified, an application of \cite[Corollary 1]{lugosi} yields
\begin{equation}
\label{c6}
\mathbf{E}_{\mathtt{M}} \to 0 \quad \mathbb{P}-a.s. \quad (\text{ as } \mathtt{M}\to \infty)\,.
\end{equation}
Finally, by \eqref{c5} and \eqref{c6} we conclude
\begin{equation*}
\lim\limits_{\mathtt{M}\to \infty} \mathbb{P}_{\mathbf{Y}^{J}}\big[ \{\mathbf{x}\in \mathbb{R}^{d}\,:\, \mathtt{diam}\bigl(\pmb{\mathcal{P}}^{J,\mathtt{M}}_{\mathtt{Ges}}[\mathbf{x}]\bigr)>\gamma \} \big] = \varepsilon \quad \mathbb{P}-a.s.\,,
\end{equation*}
which verifies condition (c) of Theorem \ref{thmlugosi}.

\end{proof}

\medskip

\begin{remark}
\label{BTCconvergence}
It is easy to verify parts (a) and (b) in Theorem \ref{thmlugosi} for the BTC-rule \ref{moddunststrategy}; to verify (c) in this case has to remain open at this stage.
\end{remark}

\section{Computational experiments}
\label{4}

All simulations are conducted via $\mathtt{PYTHON}$ (version 3.11) on a {\em HP ProBook 455 15.6 inch G9 Notebook PC (Processor: AMD Ryzen 7 5825U with Radeon Graphics, 2000 MHz)} and on a {\em Dual Xeon 6242 Workstation with 768 GB RAM}. In all computational experiments, the samples of the initial conditions are generated with Numpy’s $\mathtt{random.multivariate\_normal}$ command.
For the computational evaluation of errors in the $\mathbb{L}^{1}-$norm we use {\em Monte Carlo integration}: 
\begin{equation*}
\lVert f \rVert_{\mathbb{L}^{1}}:=\int\limits_{\mathbb{R}^{d}} \lvert f(\mathbf{x})\rvert \,\mathrm{d}\mathbf{x}\approx \sum\limits_{r=0}^{\mathtt{R}-1} \mathtt{vol}(\mathbf{R}_{r})\cdot \frac{1}{\mathsf{N}}\sum\limits_{n=1}^{\mathsf{N}} \lvert f(\mathbf{x}^{(r)}_{i})\rvert\,,
\end{equation*} 
where $\mathsf{N}\in \mathbb{N}$ is sufficiently large, and $\mathbf{x}^{(r)}_{i}$ is a randomly drawn point in the cell $\mathbf{R}_{r}$.\\
In order to approximate the $\mathbb{L}^{\infty}-$norm, we evaluate the error on each cell $\mathbf{R}_{r}$ at randomly drawn points $\mathbf{x}^{(r)}_{i}\in \mathbf{R}_{r}$ in it, i.e.,
\begin{equation*}
\lVert f \rVert_{\mathbb{L}^{\infty}}:=\esssup\limits_{\mathbf{x}\in \mathbb{R}^{d}} \lvert f(\mathbf{x})\rvert\approx \max\limits_{r=0,...,\mathtt{R}-1} \max\limits_{i=1,...,\mathsf{N}} \lvert f(\mathbf{x}_{i}^{(r)}) \rvert\,,
\end{equation*}
for some $\mathsf{N}\in \mathbb{N}$.

\medskip

The examples in Section \ref{1} showed the following properties of the density estimator in Algorithm \ref{strategy}:\\
$\bullet$ Example \ref{example1}: for $d=2$ and $\mathcal{L}^{\ast}\equiv \mathcal{L}^{\ast}(\mathbf{x})$, we see an adapted mesh that automatically adjusts to the dynamics of $j \mapsto \widehat{\mathtt{p}}_{\mathtt{Ges}}^{j,\mathtt{M}}$.\\
$\bullet$ Example \ref{example2}: for $d=5$, the second density estimator $\widehat{\mathtt{p}}_{\mathtt{BTC}}^{J,\mathtt{M}}$ is more efficient, and generates more accurate meshes and results; see also Remark \ref{BTCconvergence}.\\
$\bullet$ Example \ref{example3}: we observe corresponding behaviors for $d=8$, and an operator $\mathcal{L}^{\ast}\equiv \mathcal{L}^{\ast}(\mathbf{x})$ with non-constant coefficients.

\medskip

In the following, we ask about \q{optimal} choices for $\mathtt{M}$ and $k_{\mathtt{M}}$ in settings where $\varepsilon$ of Example \ref{example1} decreases, on using Algorithm \ref{strategy} with BTC-rule \ref{moddunststrategy}.

\begin{contiex1}
\label{contiexample1}
~\\
\begin{table}[h!]
\center
\begin{tabular}{lclclclc}
\toprule
 {\bf Setup} & $d$ & $T$ & $\tau$  & $\alpha$ & $\varepsilon$\\
\midrule
 {\bf a)}  & $2$ & $1$ & $\frac{1}{100}$  & $16$ & $4$\\
 {\bf b)}  & $2$ & $1$ & $\frac{1}{100}$  & $16$ & $\frac{1}{100}$\\
 {\bf c)}  & $2$ & $1$ & $\frac{1}{100}$  & $\frac{1}{20}$ & $1$\\
 {\bf d)}  & $2$ & $1$ & $\frac{1}{100}$  & $\frac{1}{20}$ & $\frac{1}{100}$\\
\bottomrule
\end{tabular}
\captionof{table}{Different setups for Example \ref{example1}.}
\label{tablescenarios}
\end{table}

\begin{figure}[h!]
\begin{minipage}[t]{.4\textwidth}
\centering
\includegraphics[scale=0.5]{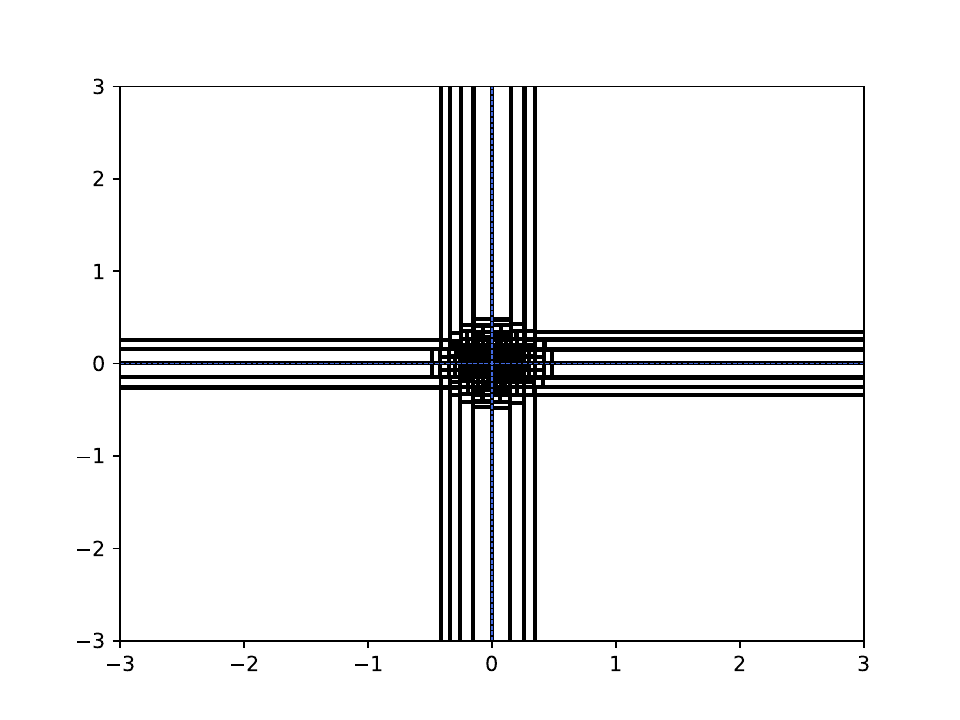}
\captionsetup{labelfont={bf}}
\subcaption{$t=0$}
\end{minipage}
\hspace{.1\linewidth}
\begin{minipage}[t]{.4\textwidth}
\centering
\includegraphics[scale=0.5]{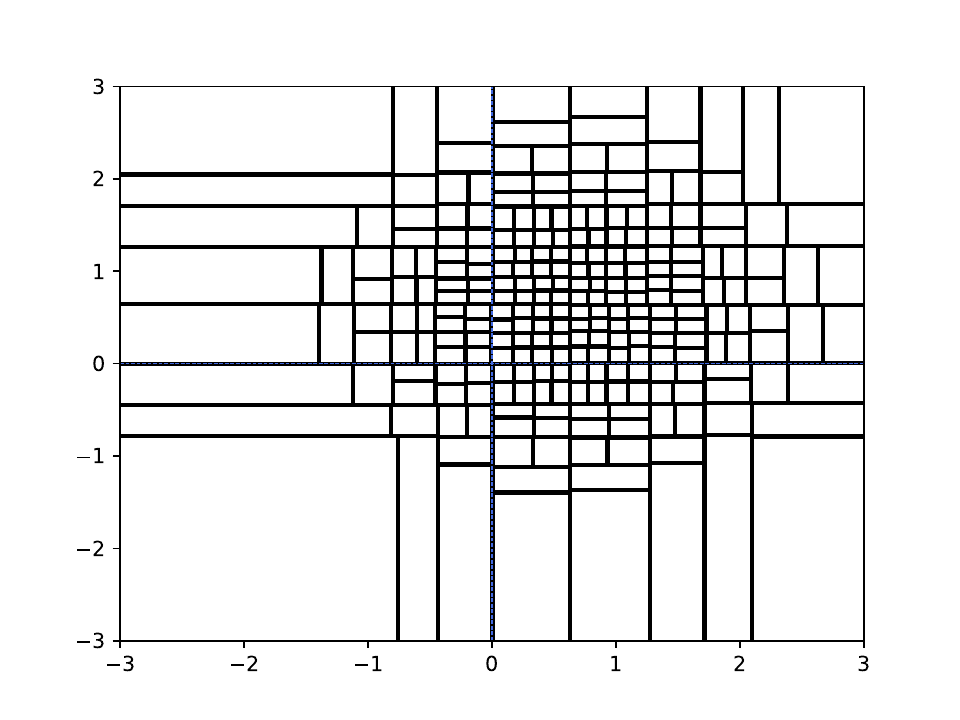}
\captionsetup{labelfont={bf}}
\subcaption{$t=T$}
\end{minipage}

\caption{{\bf Setup c):} Dynamics of the partition $\pmb{\mathcal{P}}_{\mathtt{BTC}}^{J,\mathtt{M}}$ with BTC-rule \ref{moddunststrategy} ($\mathtt{M}=2^{16}$, $k_{\mathtt{M}}=2^{8}$).}
\label{figex1_part_e}
\end{figure}

\begin{figure}[h!]
\begin{minipage}[t]{.4\textwidth}
\centering
\includegraphics[scale=0.5]{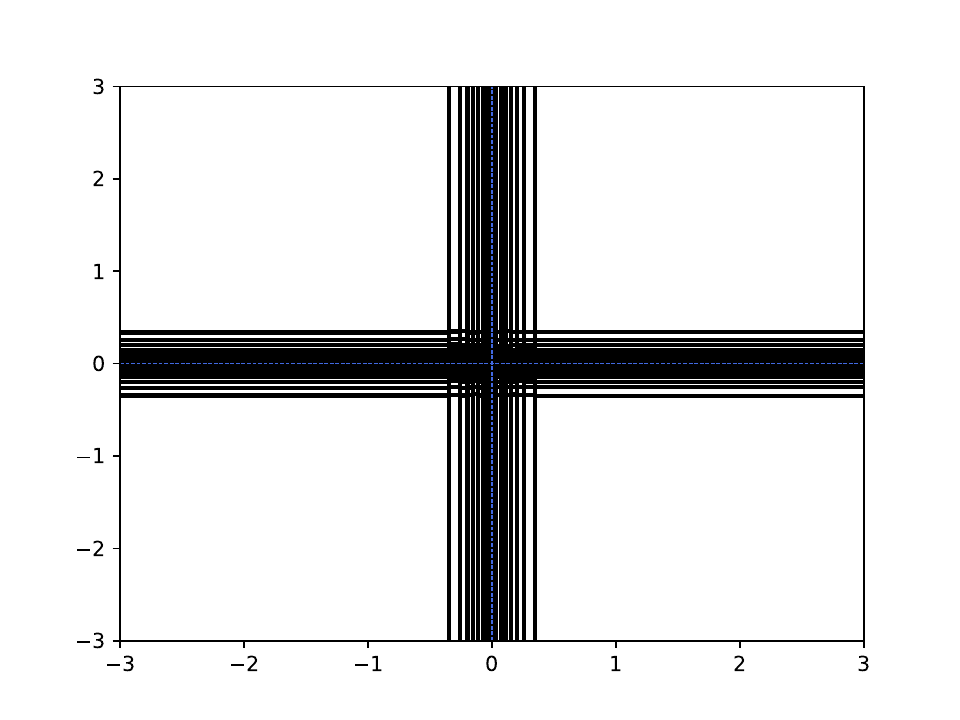}
\captionsetup{labelfont={bf}}
\subcaption{$t=0$}
\end{minipage}
\hspace{.1\linewidth}
\begin{minipage}[t]{.4\textwidth}
\centering
\includegraphics[scale=0.5]{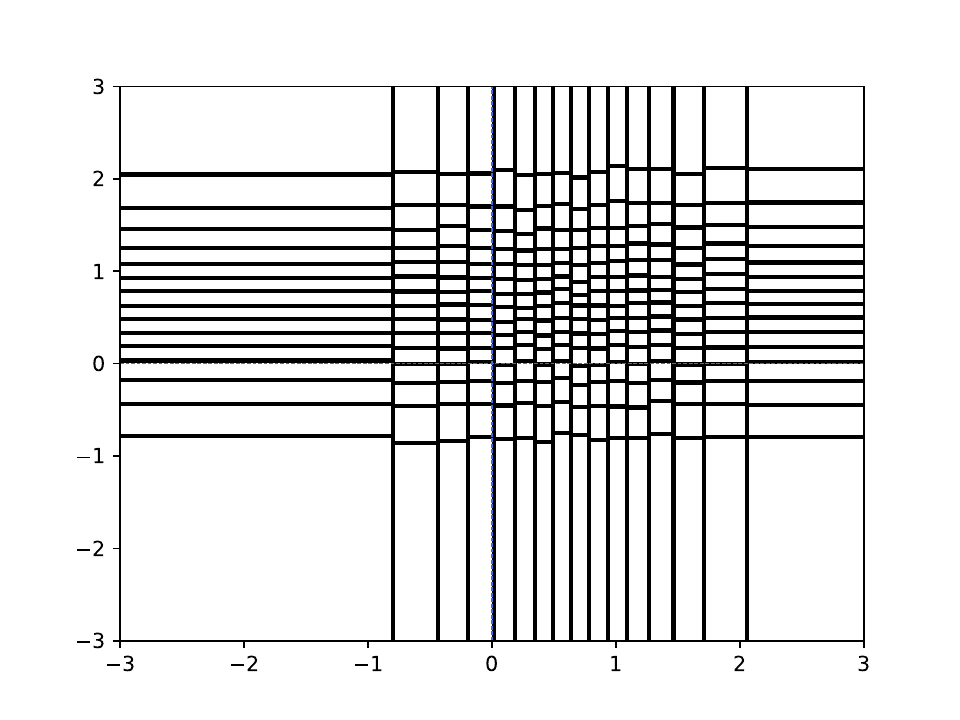}
\captionsetup{labelfont={bf}}
\subcaption{$t=T$}
\end{minipage}

\caption{{\bf Setup c):} Dynamics of the partition $\pmb{\mathcal{P}}_{\mathtt{Ges}}^{J,\mathtt{M}}$ with Gessaman's rule \ref{gessamanstrategy} ($\mathtt{M}=2^{16}$, $k_{\mathtt{M}}=2^{8}$).}
\label{figex1_part_e_gess}

\end{figure}

\begin{figure}[h!]
\begin{minipage}[t]{.4\textwidth}
\centering
\includegraphics[scale=0.6]{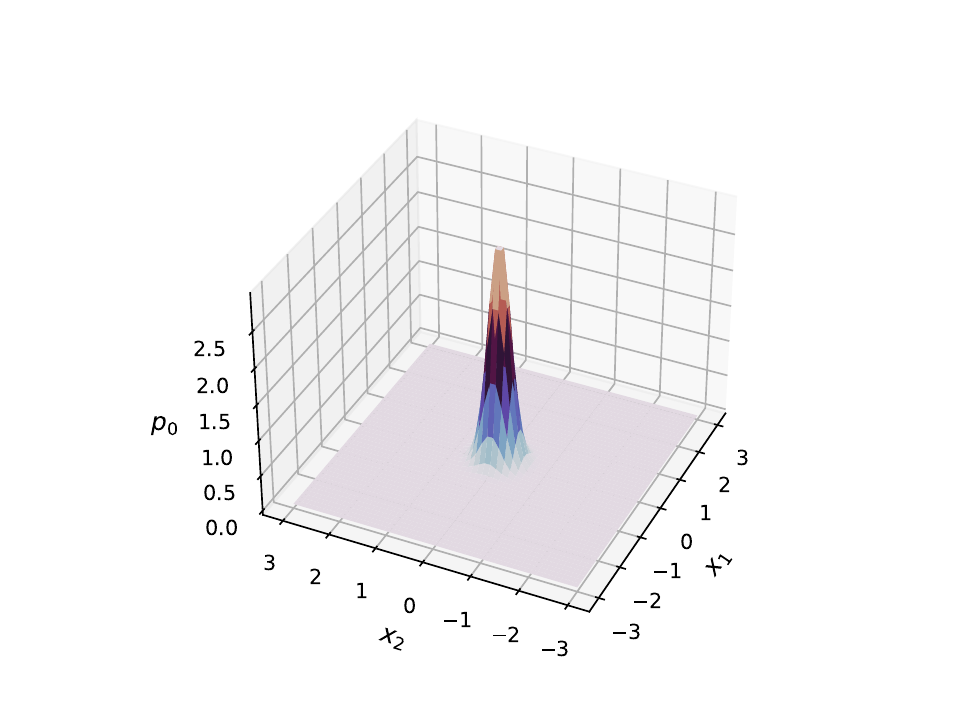}
\captionsetup{labelfont={bf}}
\subcaption{$t=0$}
\end{minipage}
\hspace{.08\linewidth}
\begin{minipage}[t]{.4\textwidth}
\centering
\includegraphics[scale=0.6]{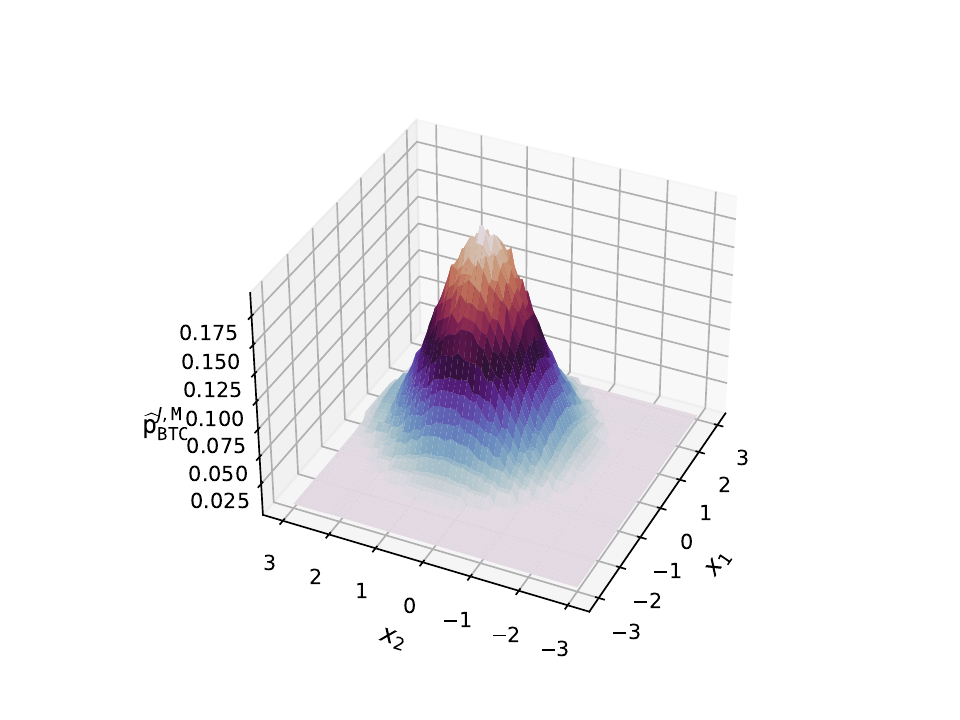}
\captionsetup{labelfont={bf}}
\subcaption{$t=T$}
\end{minipage}

\caption{{\bf Setup c):} Snapshots for $\widehat{\mathtt{p}}_{\mathtt{BTC}}^{J,\mathtt{M}}$ from Algorithm \ref{strategy} ($\mathtt{M}=2^{20}$, $k_{\mathtt{M}}=2^{10}$, $\mathtt{R}=2^{10}$).}
\label{figex1_sol_e}

\end{figure}

Solution \eqref{sol.ex1} depends on the choice of $\varepsilon$. In fact, we  have for $\mathbf{x}=\mathbf{m}_{t}$, where $p(t,\cdot)$ attains its maximum:
\begin{equation}
\label{limitexample}
\lim\limits_{t\to \infty} p(t,\mathbf{m}_{t})=\frac{1}{(2\pi \varepsilon^{2})^{\nicefrac{d}{2}}} = \mathcal{O}(\varepsilon^{-d})\,.
\end{equation}
By comparing the initial profile $p_{0}$ in \eqref{p0ex1} evaluated at $\mathbf{x}=\mathbf{0}$ and \eqref{limitexample}, the simulations show that for
\begin{itemize}
\item[(i)] $\varepsilon>\sqrt{\alpha}$, the profile of the solution \eqref{sol.ex1} (as $t\to \infty$) migrates in direction of $\pmb{\beta}$, and its \q{peak value} {\em flattens} to $\frac{1}{(2\pi \varepsilon^{2})^{\nicefrac{d}{2}}}$, while
\item[(ii)] for $\varepsilon<\sqrt{\alpha}$, the profile of the solution \eqref{sol.ex1} in fact migrates in direction of $\pmb{\beta}$ too, but its \q{peak value} grows and {\em sharpens} to $\frac{1}{(2\pi \varepsilon^{2})^{\nicefrac{d}{2}}}$; see Figure \ref{figex1_1}.
\end{itemize}

Below, we consider the following challenging scenarios presented in Table \ref{tablescenarios}.

The parameter $\alpha$ also regulates the {\em width} and {\em height} of $p_{0}$ in \eqref{p0ex1}: larger values for $\alpha$ lead to a broader peak of a lower height; whereas small values of $\alpha$ result in a thin and copped peak with a high maximum; see Figure \ref{figex1_sol_e} {\bf (a)}. Depending on the relative size of $\varepsilon$ and $\sqrt{\alpha}$, we observe a {\em flattening} resp.~{\em sharpening} of the peak structure up to a certain threshold. This {\em flattening} resp.~{\em sharpening} dynamics comes along with automatic adjustments of the underlying {\em data-dependent} partition on which the approximated solution $\widehat{\mathtt{p}}_{\mathtt{BTC}}^{J,\mathtt{M}}$ via Algorithm \ref{strategy} is realised: while {\em flattening} dynamics causes an \q{expansion/stretching} of the partition; see Figures \ref{figex1_part_e} and \ref{figex1_part_e_gess} in the setting of {\bf Setup c)}, {\em sharpening} dynamics lead to a \q{concentration/compression} of the partition; see Figure \ref{figex1} in Example \ref{example1}.

For a fixed $\mathtt{M}=2^{18}$ and varying $k_{\mathtt{M}}$, Figure \ref{figex1_errors} below shows $\mathbb{L}^{1}-$ resp.~$\mathbb{L}^{\infty}-$errors for $\widehat{\mathtt{p}}_{\mathtt{BTC}}^{J,\mathtt{M}}$ within the scenarios {\bf a)} -- {\bf d)} in Table \ref{tablescenarios}. While the $\mathbb{L}^{1}-$error curves in Figure \ref{figex1_errors} {\bf (a)} are almost the same in every scenario for varying $k_{\mathtt{M}}$, and suggest $k_{\mathtt{M}}\approx \mathtt{M}^{\frac{1}{2}-\beta}$ (for $\beta>0$ small) as optimal choice for smallest $\mathbb{L}^{1}-$errors, the $\mathbb{L}^{\infty}-$error curves in Figure \ref{figex1_errors} {\bf (b)} motivate $k_{\mathtt{M}}\approx \mathtt{M}^{\frac{3}{4}\pm \beta}$ (for $\beta>0$ small) for smallest $\mathbb{L}^{\infty}-$errors.

\begin{figure}[h!]
\begin{minipage}[t]{.4\textwidth}
\centering
\includegraphics[scale=0.5]{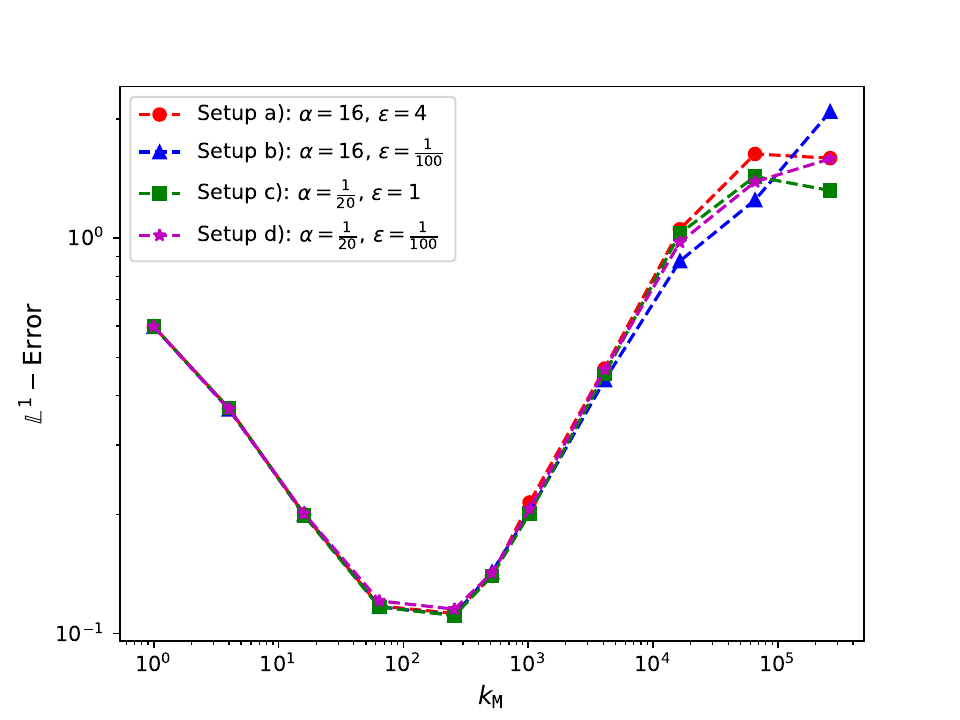}
\captionsetup{labelfont={bf}}
\subcaption{$k_{\mathtt{M}}\mapsto \lVert p(T,\cdot)-\widehat{\mathtt{p}}^{J,\mathtt{M}}_{\mathtt{BTC}}\rVert_{\mathbb{L}^{1}}$}
\end{minipage}
\hspace{.1\linewidth}
\begin{minipage}[t]{.4\textwidth}
\centering
\includegraphics[scale=0.5]{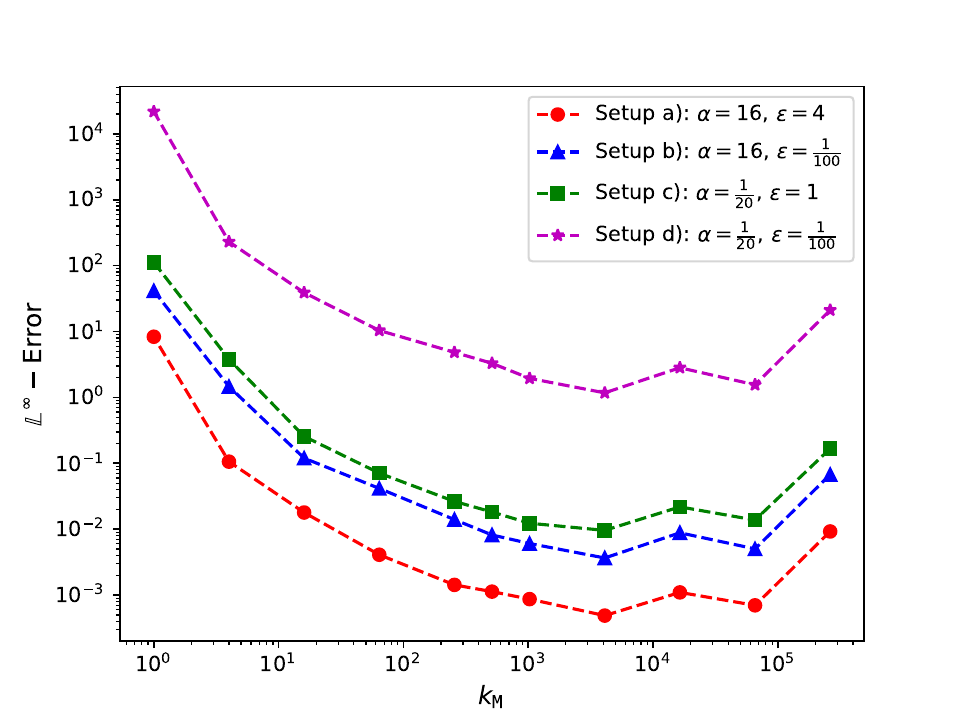}
\captionsetup{labelfont={bf}}
\subcaption{$k_{\mathtt{M}}\mapsto \lVert p(T,\cdot)-\widehat{\mathtt{p}}^{J,\mathtt{M}}_{\mathtt{BTC}}\rVert_{\mathbb{L}^{\infty}}$}
\end{minipage}

\caption{Log-plots of $\mathbb{L}^{1}-$ and $\mathbb{L}^{\infty}-$errors with fixed $\mathtt{M}=2^{18}$ and varying $k_{\mathtt{M}}$ within the settings {\bf a)}, {\bf b)}, {\bf c)} and {\bf d)}.}
\label{figex1_errors}

\end{figure}

%
%
%

\end{contiex1}

\medskip

The next example illustrates the performance of $\widehat{\mathtt{p}}_{\mathtt{BTC}}^{J,\mathtt{M}}$ in Algorithm \ref{strategy} for non-constant $\mathcal{L}^{\ast}\equiv \mathcal{L}^{\ast}(\mathbf{x})$ in \eqref{eq:operator} where the components in $\mathbf{x}\mapsto \mathbf{b}(\mathbf{x})$ and $\mathbf{x}\mapsto \pmb{\sigma}(\mathbf{x})$ are strongly coupled, and where the initial conndition $p_{0}$ is non-continuous.

\begin{example}
\label{example4}
Let $d=8$ and $T=1$. Consider \eqref{eq:fpe}, which corresponds to the associated SDE \eqref{eq:SDE} with
\begin{align*}
\mathbf{b}(\mathbf{x})={\small \begin{bmatrix}
0\\
0\\
2\\
2\\
\sin(x_{5})\\
x_{6}\\
x_{8}\\
-x_{7}\\
\end{bmatrix}}\,, \qquad 
\pmb{\sigma}(\mathbf{x})=0.1 \cdot 
{\small \begin{bmatrix}
1 & 0 & 0 & 0 & 0 & 0 & 0 & 0\\
0 & 1 & 0 & 0 & 0 & 0 & 0 & 0\\
0 & 0 & 1 & 0 & 0 & 0 & 0 & 0\\
0 & 0 & 0 & 1 & 0 & 0 & 0 & 0\\
0 & 0 & 0 & 0 & 1 & 0 & 0 & 0\\
0 & 0 & 0 & 0 & 0 & 1 & 0 & 0\\
0 & 0 & 0 & 0 & 0 & 0 & 1 & 0 \\
\arctan(x_{1})^{2} & 0 & 0 & 0 & 0 & 0 & 0 & 2+\arctan(x_{1})^{2}\\
\end{bmatrix}}\,,
\end{align*}
Let $p_{0}$ be the pdf which corresponds to the uniform distribution on the hypercube $[-0.5,0.5]^{d}$. For a fixed choice of $\tau=0.01$, $\mathtt{M}=2^{32}$ and $k_{\mathtt{M}}=2^{10}$, Figure \ref{ex4dynamics} shows approximated solution profiles $\widehat{\mathtt{p}}_{\mathtt{BTC}}^{J,\mathtt{M}}$ when restricted to different coordinate axes, indicating expectable dynamics.


\begin{figure}[h!]
\begin{minipage}[t]{.25\textwidth}
\centering
\includegraphics[scale=0.2]{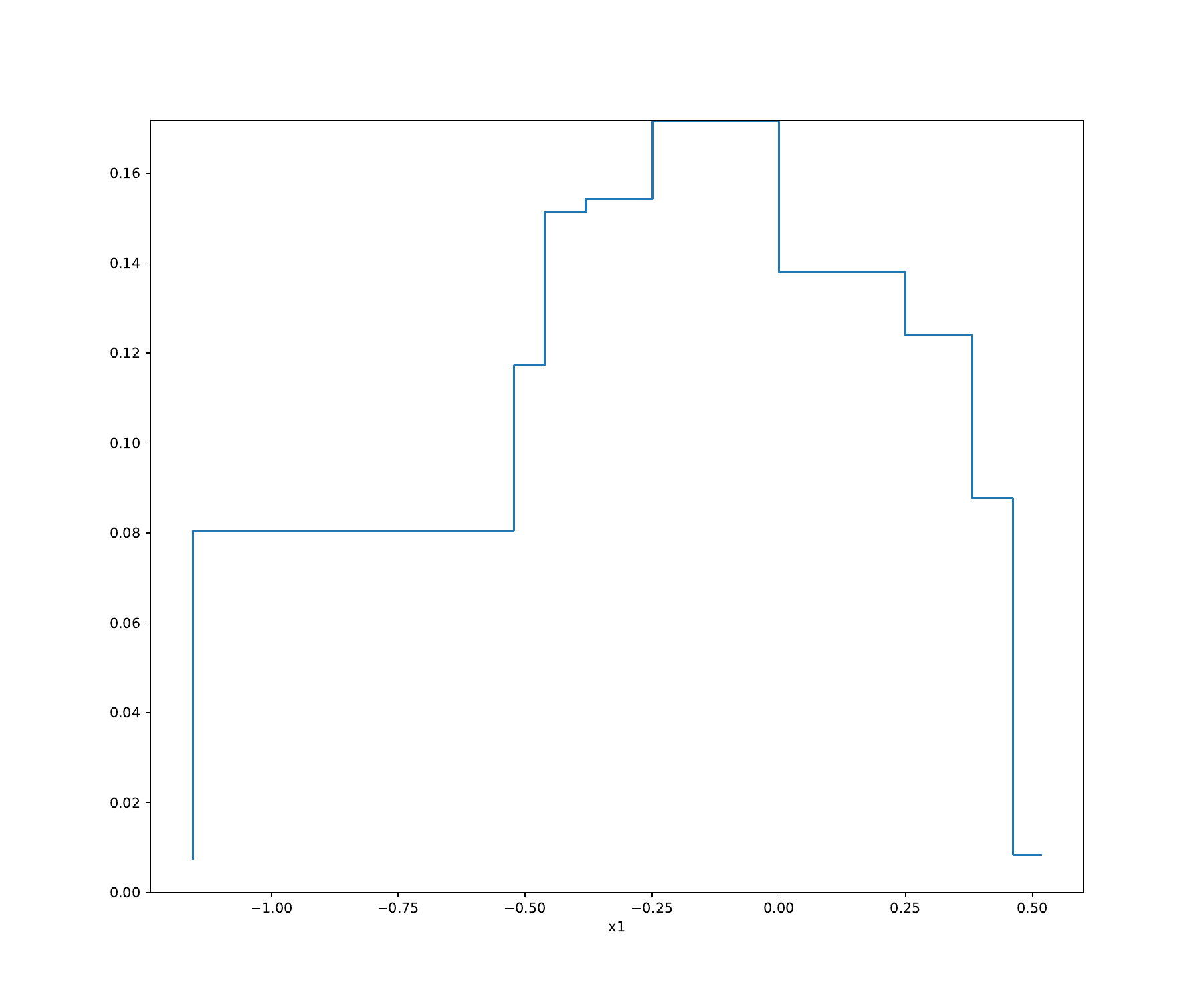}
\captionsetup{labelfont={bf}}
\subcaption{$\restr{\widehat{\mathtt{p}}^{J,\mathtt{M}}_{\mathtt{BTC}}}{\mathbf{P_{1}}}$}
\end{minipage}
\hspace{.08\linewidth}
\begin{minipage}[t]{.25\textwidth}
\centering
\includegraphics[scale=0.2]{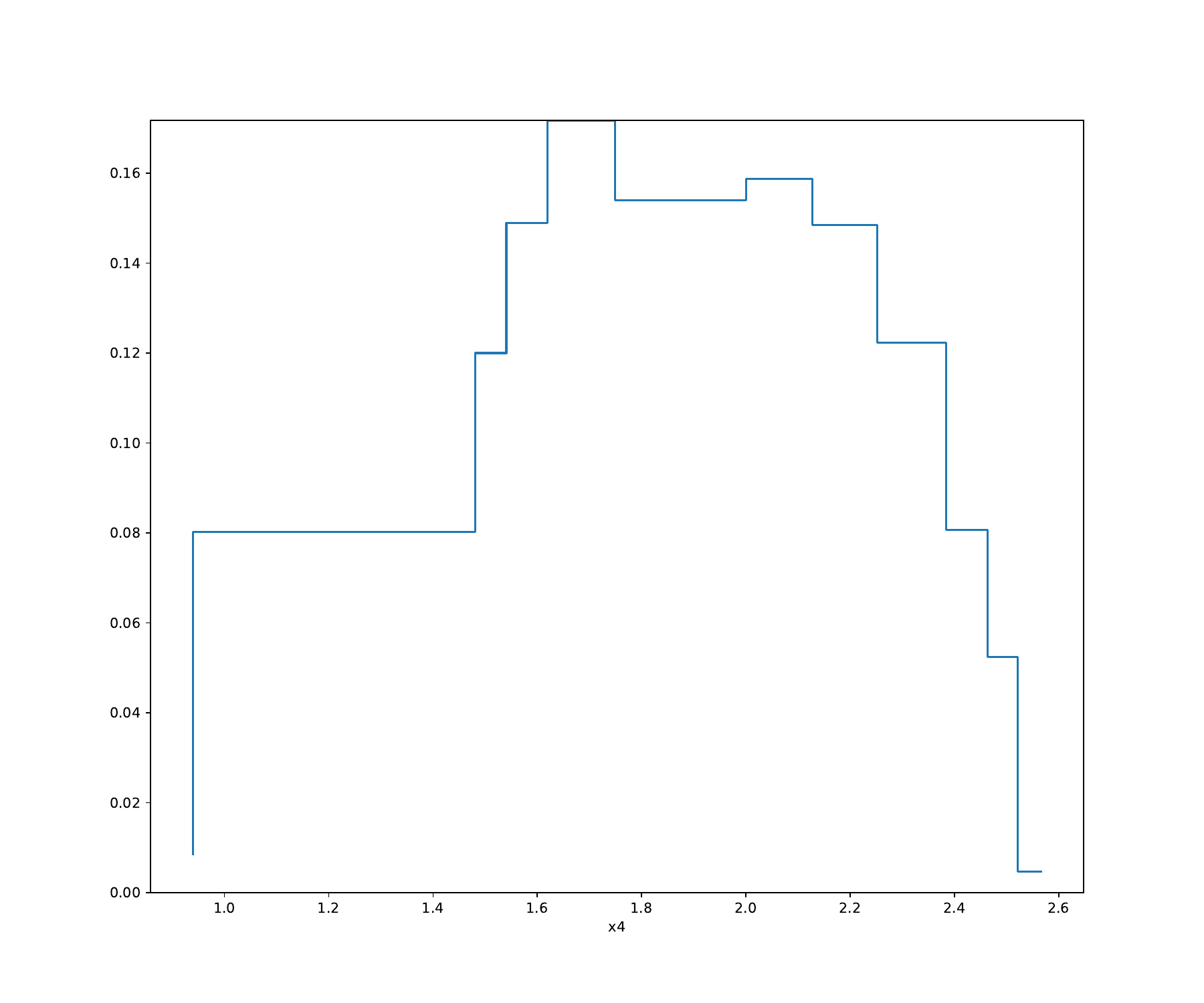}
\captionsetup{labelfont={bf}}
\subcaption{$\restr{\widehat{\mathtt{p}}^{J,\mathtt{M}}_{\mathtt{BTC}}}{\mathbf{P_{4}}}$}
\end{minipage}
\hspace{.08\linewidth}
\begin{minipage}[t]{.25\textwidth}
\centering
\includegraphics[scale=0.2]{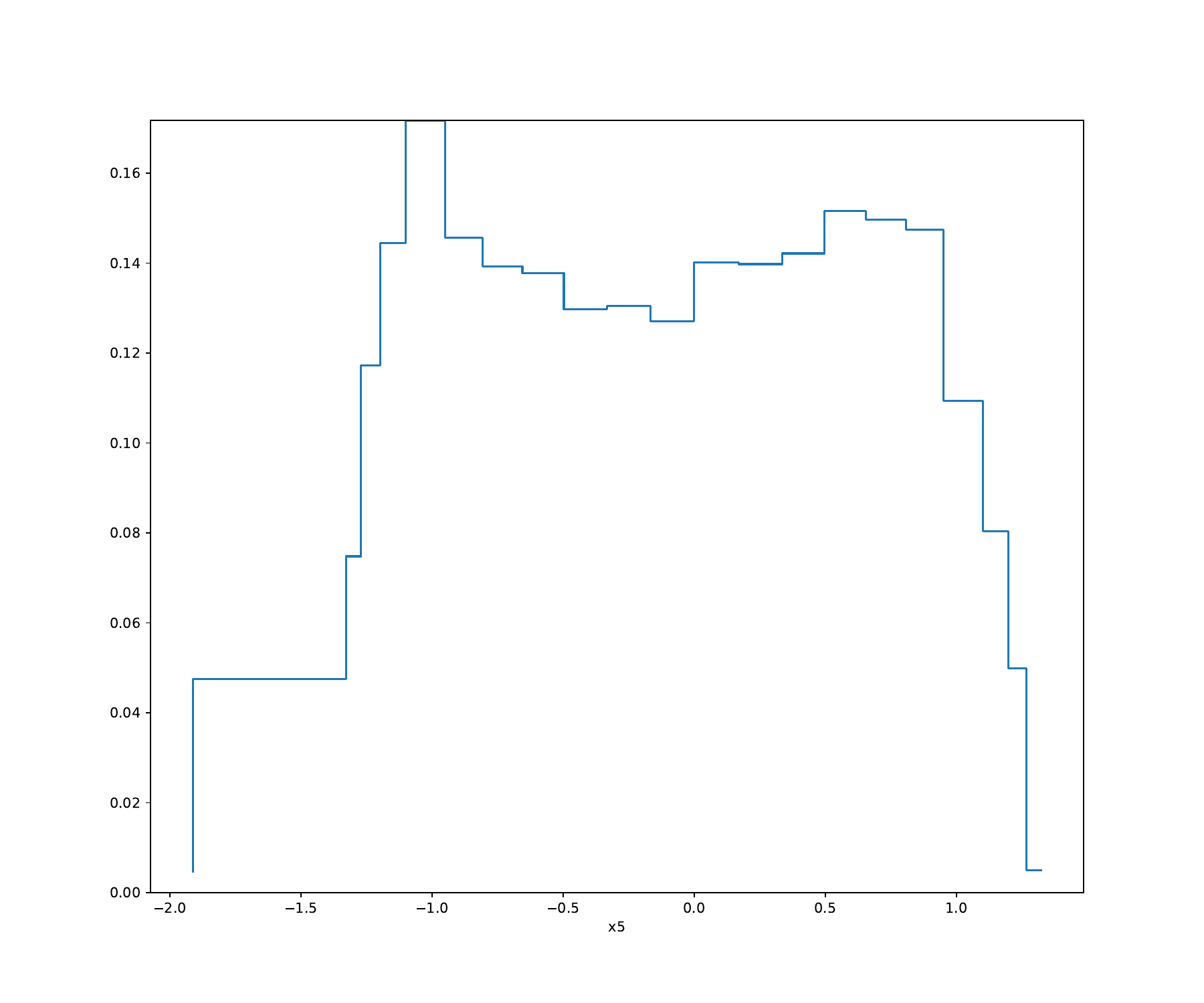}
\captionsetup{labelfont={bf}}
\subcaption{$\restr{\widehat{\mathtt{p}}^{J,\mathtt{M}}_{\mathtt{BTC}}}{\mathbf{P_{5}}}$}
\end{minipage}

\caption{ Example \ref{example4}: Approximated solution profile restricted to different coordinate axes ($\mathtt{M}=2^{32}$, $k_{\mathtt{M}}=2^{10}$).}
\label{ex4dynamics}

\end{figure}

\end{example}

\medskip

\subsection*{Acknowledgement}

We are grateful to A.~Chaudhary (U Tübingen) for helpful discussions concerning the proof of Theorem \ref{thmgessamanconvergence}.

\end{document}